\documentclass[a4paper,11pt]{amsart}
\usepackage{latexsym,psfrag,amsfonts,amsmath,amssymb,epsfig,url
}

\makeatletter
\@addtoreset{figure}{section}
\makeatother

\makeatletter
\@addtoreset{table}{section}
\makeatother

\begin{document}

\newtheorem{theorem}{Theorem}[section]
\newtheorem*{thm}{Main Theorem}
\newtheorem{conj}[theorem]{Conjecture}
\newtheorem{lemma}[theorem]{Lemma}
\newtheorem{cor}[theorem]{Corollary}
\renewcommand{\proofname}{Proof}
\newtheorem{property}[theorem]{Property}
\newtheorem{prop}[theorem]{Proposition}

\theoremstyle{definition}
\newtheorem{defin}[theorem]{Definition}
\newtheorem{question}[theorem]{Question}
\newtheorem{remark}[theorem]{Remark}
\newtheorem{example}[theorem]{Example}

\def \H{{\mathbb H}}
\def \R{{\mathbb R}}
\def \E{{\mathbb E}}
\def \Z{{\mathbb Z}}
\def \S{{\mathbb S}}
\def \G{{\mathcal G}}
\def \Gr{{\mathrm{Gr}}}
\def \Vol{{\mathrm{Vol}}}
\def \l{\langle }
\def \r{\rangle }
\def \[{[ }
\def \]{] }
\def \d{D\,}
\def \sign{\text{\,sign\,}}
\def \conv{\text{\,conv\,}}
\def \wt{\widetilde}
\def \wh{\widehat}
\def \a{\alpha}
\def \b{\beta}
\def \dim{\text{\,dim}}
\def \u{\underline}

\newcommand{\arcsinh}{\mathop{\mathrm{arcsinh}}\nolimits}
\newcommand{\vn}{\mathop{\mathrm{int}}\nolimits}
\newcommand{\rel}{\mathop{\mathrm{rel\:int}}\nolimits}
\newcommand{\w}{\widetilde }
\renewcommand{\o}{\overline }

\author{Anna Felikson}
\address{Independent University of Moscow, B. Vlassievskii 11, 119002 Moscow, Russia}
\curraddr{School of Engineering and Science, Jacobs University Bremen, P.O. Box 750 561, 28725 Bremen, Germany}
\email{a.felikson@jacobs-university.de}
\thanks{Research was partially supported by grants RFBR 07-01-00390-a and RFBR 11-01-00289-a}

\author{Pavel Tumarkin}
\address{Department of Mathematical Sciences, Durham University, Science Laboratories, South Road, Durham, DH1 3LE, UK}
\email{pavel.tumarkin@durham.ac.uk}

\title{Essential hyperbolic Coxeter polytopes}

\maketitle

\begin{abstract} 
We introduce a notion of essential hyperbolic Coxeter polytope as a polytope which fits some minimality conditions. The problem of classification of hyperbolic reflection groups can be easily reduced to classification of essential Coxeter polytopes. We determine a potentially large combinatorial class of polytopes containing, in particular, all the compact hyperbolic Coxeter polytopes of dimension at least $6$ which are known to be essential, and prove that this class contains finitely many polytopes only. We also construct an effective algorithm of classifying polytopes from this class, realize it in four-dimensional case, and formulate a conjecture on finiteness of the number of essential polytopes.  
\end{abstract}

\tableofcontents

\section{Introduction}

\thispagestyle{empty}
The main goal of this paper is to provide a classification algorithm for a large class of compact hyperbolic Coxeter polytopes. A Coxeter polytope in hyperbolic space is a convex fundamental domain for reflection group, i.e. a polytope whose all dihedral angles are integer parts of $\pi$. For a compact hyperbolic Coxeter polytope $P$ we denote by $d$, $n$ and $p$ the dimension, the number of codimension one faces ({\it facets} in the sequel), and the number of pairs of disjoint facets respectively.     
Let ${\mathcal P}$ be the set of polytopes satisfying the following condition: 

$$P\in {\mathcal P}\quad \text{if}\quad d\ge 4 \quad \text{and}\quad p\le n-d-2$$

\begin{thm}[Theorem~\ref{main}]
The set  ${\mathcal P}$ is finite.
All polytopes of this set can be listed by a finite algorithm
provided in Section~\ref{algorithm}.
\end{thm}

\begin{remark}
\label{nonew}
The algorithm was implemented for polytopes of dimension $4$ (see Section~\ref{dim4}), and no new polytopes were produced. In fact, it is not clear whether the implementation of the algorithm will lead to new polytopes or will show that all the polytopes from the set ${\mathcal P}$ are already listed. 

\end{remark}

In general, the number of compact hyperbolic Coxeter polytopes is infinite. In~\cite{M}, Makarov constructed infinite series of polytopes in dimensions $4$ and $5$. In~\cite{All}, Allcock used the $6$-dimensional polytope constructed by Bugaenko~\cite{Bu2} to present an infinite series of $6$-polytopes. Another source of examples is given by right-angled compact Coxeter polytopes, they exist in hyperbolic spaces of dimension at most $4$ (see~\cite{PV}). 

Recently Nikulin~\cite{N} and Agol, Belolipetsky, Storm, Whyte~\cite{ABSW} proved that the number of maximal arithmetic hyperbolic reflection groups is finite. A similar statement concerning general reflection groups (i.e., with maximality or arithmeticity condition dropped) does not hold: examples by Makarov provide fundamental domains of infinite number of non-commensurable reflection groups. 
However, all known examples of infinite series of polytopes (in dimension at least $4$) can be obtained from a finite number of polytopes by a composition of two operations: taking a fundamental domain of a finite index reflection subgroup of the corresponding reflection group, or gluing two Coxeter polytopes along congruent facets. This gives rise to a notion of {\it essential} hyperbolic Coxeter polytope which is minimal with respect to operations above (see Section~\ref{essential} for precise definitions). A classification problem of compact hyperbolic Coxeter polytopes can be easily reduced to a classification of essential polytopes.         

In dimensions $2$ and $3$ compact hyperbolic Coxeter polytopes are completely classified by Poincar\'e~\cite{P} and Andreev~\cite{An}, so in this paper we restrict our attention to polytopes in the spaces of dimension at least $4$. As we have mentioned above, the number of known essential compact hyperbolic Coxeter polytopes is finite (all known polytopes can be obtained by gluing several copies of finite number of polytopes). Moreover, all polytopes of dimension greater than five that are known to be essential belong to ${\mathcal P}$. In Section~\ref{essential} we discuss properties of essential polytopes and present a list of them. While proving that a polytope is essential, the main tool is 
the result of~\cite{subgr} which, roughly speaking, states that gluing operations cannot decrease the number of facets.  

A possible counterpart of the result by Nikulin and Agol, Belolipetsky, Storm, Whyte in non-arithmetic case would be a finiteness of the number of essential polytopes, which we conjecture in Section~\ref{essential}.

\bigskip

All the remaining sections are devoted to the proof of the Main Theorem.
Here is the plan of the proof. First, we split the set of polytopes ${\mathcal P}$ in the following way:
$${\mathcal P}=\bigsqcup\limits_{d\ge 4}{\mathcal P}_{d},\qquad {\mathcal P}_d=\bigsqcup\limits_{n\ge d+1} {\mathcal P}_{(d,n)}$$
where ${\mathcal P}_{(d,n)}$ consists of polytopes from ${\mathcal P}$ of dimension $d$ with $n$ facets. Further, each ${\mathcal P}_{(d,n)}$ can be presented as 
$${\mathcal P_{(d,n)}}=\bigcup\limits_{k\ge 2} {\mathcal P}_{(d,n),k}$$
where ${\mathcal P}_{(d,n),k}$ consists of polytopes from ${\mathcal P}_{(d,n)}$ with dihedral angles not less than $\pi/k$.   

For given $d,n$ and $k$ there are finitely many combinatorial types of polytopes and finitely many ways to assign angles between facets, which implies that ${\mathcal P}_{(d,n),k}$ is finite due to the following fact proved by Andreev~\cite{An0}: an acute-angled polytope $P$ is completely determined by its dihedral angles. Our aim is to prove that for given $d$ and $n$ the set ${\mathcal P}_{(d,n),k}$ is empty for large $k$ (which will imply that ${\mathcal P}_{(d,n)}$ is finite), we do it in Section~\ref{multiplicity}. The key tool here is the trick allowing us to reduce dimension. It is based on the following result due to Borcherds~\cite{Bo}: if a face $F$ of codimension $2$ is an intersection of two facets composing a dihedral angle less than $\pi/3$, then $F$ itself is a Coxeter polytope.

Then we show (Section~\ref{fin}) that for given $d$ the set ${\mathcal P}_{(d,n)}$ is empty for large $n$ (which will imply that ${\mathcal P}_{d}$ is finite). In~\cite{abs}, Vinberg proved that the dimension of a compact hyperbolic Coxeter polytope does not exceed $29$. Therefore, we need to consider sets ${\mathcal P}_{d}$ for a finite number of dimensions only, so ${\mathcal P}$ is finite.

In Section~\ref{cox} we recall basic facts about Coxeter polytopes and their facets, and reproduce the notation introduced in~\cite{nodots}; in Section~\ref{tools} we describe the technique of local determinants and prove several technical lemmas. 
In Section~\ref{algorithm} we construct an algorithm which allows us to list all the polytopes of ${\mathcal P}$. Section~\ref{dim4} is devoted to investigation of $4$-dimensional case.

\medskip

We would like to thank the University of Fribourg, where the most part of the work was carried out, for hospitality and nice working atmosphere,
we thank R.~Kellerhals for invitation, numerous stimulating discussions, and for partial support by SNF projects 200020-113199 and
200020-121506/1.
We are grateful to V.~Emery for communicating unpublished results of~\cite{BE} and to A.~Postnikov for the idea of the proof of Lemma~\ref{esm}.
We also thank IHES (where the previous version of the paper was prepared) for hospitality and excellent working conditions, and the referee for valuable comments.

\section{Coxeter polytopes and their facets}
\label{cox}
In this section we list the essential facts about Coxeter polytopes and their faces.
We mainly follow~\cite{V1} and~\cite{29}, see also~\cite{nodots}.

In what follows we write {\it $d$-polytope} instead of
``$d$-dimensional polytope'', {\it $k$-face} instead of ``$k$-dimensional
face'' and {\it facet} instead of ``face of codimension one''.

\subsection{Coxeter diagrams}
\label{sec-d}
An abstract {\it Coxeter diagram} $\Sigma$ is a finite $1$-dimensional
simplicial complex with weighted edges, where weights $w_{ij}$ are
positive, and if $w_{ij}<1$ then $w_{ij}=\cos \frac{\pi}{k}$ for
some integer $k\ge 3$.
A {\it sub\-dia\-gram} of $\Sigma$ is a subcomplex with the same weights
as in $\Sigma$.
The {\it order}  $|\Sigma|$ is the number of nodes
of the diagram $\Sigma$.

If $\Sigma_1$ and $\Sigma_2$ are sub\-dia\-grams of an abstract
Coxeter diagram $\Sigma$,
we denote by $\l \Sigma_1,\Sigma_2\r$ a sub\-dia\-gram of $\Sigma$
spanned by the nodes of $\Sigma_1$ and $\Sigma_2$.

Given an abstract Coxeter diagram $\Sigma$ with nodes
$v_1,\dots,v_n $ and weights $w_{ij}$, we construct a symmetric
$n\times n$ matrix $\Gr(\Sigma)=(c_{ij})$, where $c_{ii}=1$,
$c_{ij}= -w_{ij}$ if  $v_i$ and $v_j$ are  adjacent, and
$c_{ij}=0$ otherwise.
By determinant, rank and signature of $\Sigma$ we mean the
determinant, the rank and the signature of $\Gr(\Sigma)$.

We can suppress the weights but indicate the same information by
labeling the edges of a Coxeter diagram in the following way: if
the weight $w_{ij}$ equals $\cos\frac{\pi}{m}$, $v_i$ and $v_j$
are joined by an $(m-2)$-fold edge or a simple edge labeled by
$m$; if $w_{ij}=1$, $v_i$ and $v_j$ are joined by a bold edge; if
$w_{ij}>1$, $v_i$ and $v_j$ are joined by a dotted edge labeled
by $w_{ij}$ (or without any label).


By a {\it multiple} edge we mean an edge of weight
$\cos\frac{\pi}{m}$ for $m\ge 4$. By a {\it multi-multiple}
edge we mean an edge of weight $\cos\frac{\pi}{m}$ for $m\ge 6$.

\medskip

An abstract Coxeter diagram $\Sigma$ is {\it elliptic} if
$\Gr(\Sigma)$ is positive definite. 
A diagram $\Sigma$ is {\it parabolic} if any indecomposable component of $\Gr(\Sigma)$ is
degenerate and positive semidefinite; a connected diagram $\Sigma$
is a {\it Lann\'er} diagram if $\Gr(\Sigma)$ is indefinite but any
proper sub\-dia\-gram of $\Sigma$ is elliptic; 
$\Sigma$  is {\it hyperbolic} if its negative inertia index
is equal to $1$.

\medskip

The list of connected elliptic and parabolic diagrams with
their standard notation is contained in~\cite[Tables~1,2]{29}.
See also~\cite[Table~3]{29} for the list of Lann\'er diagrams.
We need the following property of these lists: any Lann\'er diagram 
of order greater than $2$ contains a multiple edge.

\bigskip

It is convenient to describe Coxeter polytopes by their Coxeter diagrams.
Let $P$ be a Coxeter polytope with facets $f_1,\dots,f_r$.
The Coxeter diagram $\Sigma(P)$ of the polytope $P$
is a diagram with nodes $v_1,\dots,v_r$; two edges $v_i$ and
$v_j$ are not joined if the hyperplanes spanned by $f_i$ and $f_j$
are orthogonal;
 $v_i$ and $v_j$ are joined by an edge with weight
$$w_{ij}=
\begin{cases}
\cos \frac{\pi}{k},&\text{ if  $f_i$ and $f_j$ form a  dihedral angle
$\frac{\pi}{k}$;}\\
1,&\text{ if $f_i$ is parallel to $f_j$;}  \\
\cosh \rho,&\text{ if $f_i$ and $f_j$ diverge and $\rho$ is the distance
between $f_i$ and $f_j$.}
\end{cases}
$$

\noindent
If $\Sigma=\Sigma(P)$, then $\Gr(\Sigma)$ coincides with the Gram matrix
of outer unit normals to the facets of $P$ (referring to the
standard model of hyperbolic $d$-space in $\R^{d,1}$).

It is shown in~\cite{V1} that a Coxeter diagram $\Sigma(P)$
of a compact $d$-dimensi\-onal hyperbolic polytope $P$ is a
connected diagram of signature $(d,1)$ without parabolic
sub\-dia\-grams.
In particular, $\Sigma(P)$ contains no bold edge, and any indefinite
sub\-dia\-gram is hyperbolic and contains a Lann\'er diagram. 
Moreover, it is shown there that any compact hyperbolic Coxeter $d$-polytope $P$ is
simple (i.e., every face of dimension $k$ is an intersection of exactly $d-k$ facets), and elliptic sub\-dia\-grams of $\Sigma(P)$ are in
one-to-one correspondence with faces of $P$: a $k$-face $f$
corresponds to an elliptic sub\-dia\-gram of order $d-k$ whose
nodes correspond to the facets of $P$ containing $f$.

\begin{prop}[\cite{nodots}, Lemma 1]
\label{same}
Let $\Sigma(P)$ be a Coxeter diagram of a hyperbolic Coxeter $d$-polytope
$P$ of finite volume. Then no proper sub\-dia\-gram of $\Sigma(P)$ is a diagram
of a hyperbolic Coxeter $d$-polytope of finite volume.
\end{prop}

We denote by $\mathcal{C}_d$ the set of all Coxeter diagrams of compact hyperbolic Coxeter 
$d$-polytopes, and by $\mathcal{C}_{d,k}$ the set of Coxeter diagrams of compact hyperbolic 
Coxeter $d$-polytopes containing no edges of multiplicity greater than $k-2$. 

\subsection{Faces of Coxeter polytopes}

Let $P$ be a compact hyperbolic Coxeter $d$-polytope, and denote by $\Sigma(P)$
its Coxeter diagram. Let $S_0$ be an elliptic sub\-dia\-gram of
$\Sigma(P)$.
By~\cite[Theorem~3.1]{V1}, $S_0$ corresponds to a face of $P$ of dimension
$d-|S_0|$. Denote this face by $P(S_0)$. $P(S_0)$ itself is an
acute-angled polytope, but it might not be a Coxeter polytope.
Borcherds proved the following
sufficient condition for $P(S_0)$ to be a Coxeter polytope.

\begin{theorem}[\cite{Bo}, Example 5.6]
\label{bor}
Suppose $P$ is a Coxeter polytope with diagram $\Sigma(P)$,
and $S_0\subset \Sigma(P)$ is an elliptic sub\-dia\-gram that has no
$A_n$ or $D_5$ component. Then $P(S_0)$ itself is a Coxeter polytope.
\end{theorem}

Facets of $P(S_0)$ correspond to those nodes
that together with $S_0$ comprise an elliptic subdiagram of $\Sigma(P)$.

Let $w\in \Sigma(P)$ be a {\it neighbor} of $S_0$,
so that $w$ attaches to $S_0$ by some edges.
We call $w$ a {\it good neighbor} if $\l S_0, w\r$ is an
elliptic diagram, and {\it bad} otherwise.
We denote by $\overline S_0$ the sub\-dia\-gram of $\Sigma(P)$
consisting of nodes corresponding to facets of $P(S_0)$. The
diagram $\overline S_0$ is spanned by good neighbors of $S_0$ and
by all vertices not joined with $S_0$ (in other words, $\overline
S_0$ is spanned by all nodes of $\Sigma(P)\setminus S_0$ except
bad neighbors of $S_0$). If $P(S_0)$ is a Coxeter polytope, denote
its Coxeter diagram by $\Sigma_{S_0}$.

In~\cite[Theorem 2.2]{All}, Allcock gives a receipt how to compute the dihedral angles of $P(S_0)$. We need the following partial case of this result:

\begin{prop}[\cite{All}]
\label{cor_All}
Suppose that $P(S_0)$ is a Coxeter polytope.
If $S_0$ has no good neighbors then $\overline S_0=\Sigma_{S_0}$.
In particular, this always holds for $S_0=G_2^{(m)}$ where $m\ge 6$.
\end{prop}

Proposition~\ref{cor_All} allows us to identify diagrams  $\overline S_0$ and  $\Sigma_{S_0}$ for $S_0$ being multi-multiple edge.

\section{Technical tools}
\label{tools}

In this section we recall the technique of local determinants derived in~\cite{abs} and prove several properties of determinants of hyperbolic Coxeter diagrams. 

Let $\Sigma$ be a Coxeter diagram, and let $T$ be a sub\-dia\-gram
of $\Sigma$ such that $\det(\Sigma\setminus T)\ne 0$.
A {\it local determinant} of $\Sigma$ on a sub\-dia\-gram $T$ is
$$\det(\Sigma,T)=\frac{\det \Sigma}{\det(\Sigma\!\setminus\! T)}.$$

The following result is the main tool in the consideration of multiplicities of edges. 

\begin{prop}[\cite{abs}, Proposition~12]
\label{loc_sum}
If a Coxeter diagram $\Sigma$ consists of sub\-dia\-grams
$\Sigma_1,\Sigma_2,\dots,\Sigma_l$ having a unique vertex $v$ in common,
and no vertex of $\Sigma_i\setminus v$ attaches to
$\Sigma_j\setminus v$ for $i\ne j$, then
$$ \det(\Sigma,v)=\det(\Sigma_1,v)+\cdots+\det(\Sigma_l,v)-(l-1).$$
\end{prop}

\noindent
Now we prove several properties of determinants of hyperbolic diagrams. The first one concerns monotonicity.

\begin{lemma}
\label{monoton}
Let $S$ be an elliptic diagram and $x\notin S$ be any node such that 
$\l S,x\r$ is hyperbolic.
Then $\det(\l S,x\r)$ is a strictly decreasing function on label of each
edge joining $x$ with $S$. 

Moreover,  $\det(\l S,x\r)$ is a strictly decreasing function on label of each
edge $\l v_1,v_2\r$ such that $\l S,x\r \setminus \l v_1,v_2\r$ is an elliptic diagram. 
\end{lemma}

\begin{proof}
Denote $x=v_1$,  $S=\l v_2,\dots,v_d\r$. Let $a_{ij}$ be the negative of the label of the edge $\l v_i,v_j\r$, $a_{ii}=1$, so that  $\det(\l S,x\r)=\det(a_{ij})$.   
We will prove that  $\det(\l S,x\r)$ is a strictly decreasing function on label $X=-a_{12}=-a_{21}$
of the edge $\l v_1,v_2\r$.

Consider a standard expanding of $\det(\l S,x\r)$ into $d!$ summands $s_i$
(each obtained as a product of $d$ matrix elements). Then

$$
\det(\l S,x\r)=aX^2+bX+c
$$

It is easy to see that $$a=-\det(\l  S,v_1\r \setminus \{ v_1,v_2\})=-\det(S\setminus v_2)<0$$
(since $S$ is an elliptic diagram). Recall that $X\ge 0$ as a label of an edge of a Coxeter diagram,
so it is sufficient to prove that  $b\le 0$.

Each of $d!$ summands of $\det(\l S,x\r)$ can be written as  
$$s_i=\prod\limits_{r=1}^{k_i} g_r \cdot \prod\limits_{r=1}^{l_i}a_{rr}
\eqno(*)
$$
where 
$$g_r=(-1)^{q_r+1}a_{j_1j_2} a_{j_2j_3} a_{j_3j_4} \dots  a_{j_{q_r}j_1}$$
are products of all the weights of edges of pairwise disjoint cycles in $\l S,x\r$ (following~\cite{abs}, we call $g_r$ {\it cyclic products}), $a_{rr}=1$, and
$$l_i=(d+1)-\sum\limits_{r=1}^{k_i}q_r$$
Notice that $a_{ij}\le 0$ for $i\ne j$ implies $g_r\le 0$.

If a summand $s_i$ contributes to a linear part of  $aX^2+bX+c$ then
$s_i$ contains exactly one of $a_{12}$ and $a_{21}$.
We will prove that the sum of the terms containing $a_{12}$ or $a_{21}$ is non-positive.

Let $g_{r_0}=(-1)^{q_{r_0}+1}a_{j_1j_2} a_{j_2j_3} \dots a_{j_{q_{r_0}}j_1}$ be a cyclic product containing $a_{12}$, denote by $L(g_{r_0})$ the sum of all terms $s_i$ 
containing $g_{r_0}$. It is easy to see that 
$$ L(g_{r_0})= g_{r_0} \det(\l S,x\r\setminus I) , $$ 
where $I=\{ v_{j_1},v_{j_2},\dots,v_{j_{q_{r_0}}} \}$.
Since   $g_{r_0}$ contains $a_{12}$ or $a_{21}$, $I$ contains both $v_1$ and $v_2$, which imply that the diagram $\l S,x\r\setminus I$ is elliptic, so $\det(\l S,x\r\setminus I)>0$.
Hence, $L(g_{r_0})\le 0$. 
Since at most one of $g_r$ in the decomposition ($*$) contains $a_{12}$ or $a_{21}$,
each of $s_i$ contributes to at most one of $L(g_{r})$,
so, $$bX=\sum\limits_r L(g_{r})$$ which is negative.
In view of the inequality $X\ge 0$ we get $b\le 0$, and the first statement of the lemma is proved. 

Now one can notice that we used only the fact that $\l S ,x\r\setminus \{ v_1,v_2\}$ is an elliptic diagram,
so the second statement holds by the same argument.
\end{proof}

\noindent
The next two lemmas provide upper bounds for determinants of some hyperbolic diagrams.

\begin{lemma}
\label{separate1}
There exists a constant $M_1(d)$ which depends only on the dimension $d$ 
such that
for any diagram $\l S,x\r$ satisfying
\begin{itemize}
\item[1)]
$S$ is an elliptic diagram of order $d$,  $\l S,x\r$ is hyperbolic,
\item[2)]
either $|S|\le 3$, or $S$ contains no multi-multiple edges,
\item[3)]
$\l S,x \r$ contains no dotted edges,
\end{itemize}
the inequality
$$
\det(\l S,x \r)\le-M_1(d)<0 
$$
holds.
\end{lemma}

\begin{proof}
Denote by $U$ the diagram obtained from $\l S,x\r$ by replacing all the edges joining $x$ with $S$ with labels higher than $7$ by edges labeled by $7$.
It is easy to see that each hyperbolic subdiagram of $\l S,x\r$ remains hyperbolic after this change, and each elliptic subdiagram remains elliptic, so $\det(U)<0$.
By Lemma~\ref{monoton}, $\det(\l S,x\r)$ is a decreasing function on a multiplicity of each edge joining $x$ with $S$.
Hence, $\det(\l S,x\r)<\det(U)$.
If $S$ contains no multi-multiple edges, then given a dimension $d$ there are finitely many possibilities for $U$ (since $|U|=d+1$),
so we may take for $M_1(d)$ any positive number smaller than  minimum of  $|\det(U)|$.

If $S$ is a diagram of order $2$ or $3$ containing multi-multiple edges, 
then the second statement of Lemma~\ref{monoton} implies that $\det(\l S,x\r)$ is a decreasing function on a multiplicity of each edge
of $\l S,x\r$ (here we use the third assumption of the lemma).
We replace each multi-multiple edge of $\l S,x\r$ with label higher than $7$ by an edge labeled by $7$,
denote the obtained diagram by $U'$ and obtain finitely many possibilities for $U'$,
which finishes the proof.
\end{proof}

Denote by $L_{p,q,r}$ a Lann\'er diagram of order $3$ containing sub\-dia\-grams of types $G_2^{(p)}$,
$G_2^{(q)}$ and $G_2^{(r)}$ (see Fig.~\ref{Lanner}).
Let $v$ be the vertex of $L_{p,q,r}$ that does not belong to $G_2^{(r)}$.
Denote by $\d(p,q,r)$ the local determinant $\det(L_{p,q,r},v)$.
It is easy to compute (see e.g.~\cite{abs}) that 
$$\d(p,q,r)=1-\frac{\cos^2(\frac{\pi}{p})+\cos^2(\frac{\pi}{q})+2\cos(\frac{\pi}{p})\cos(\frac{\pi}{q})\cos(\frac{\pi}{r})}{\sin^2(\frac{\pi}{r})}$$
The explicit expression for $\d(p,q,r)$ shows that
$|\d(p,q,r)|$ is an increasing function on each of $p,q,r$ tending to
infinity while $r$ tends to infinity. In particular, one can check that $|\d(p,q,r)|\ge|\d(2,3,7)|\approx 0.329$.

\begin{figure}[!h]
\begin{center}
\psfrag{p}{ $p$}
\psfrag{q}{ $q$}
\psfrag{r}{ $r$}
\psfrag{v}{ $v$}
\epsfig{file=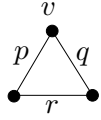,width=0.1\linewidth}
\caption{A diagram $L_{p,q,r}$}
\label{Lanner}
\end{center}
\end{figure}

The next lemma shows that any $\Sigma\in \mathcal{C}_d$ contains a certain hyperbolic subdiagram with determinant of relatively large absolute value (see Section~\ref{sec-d} for definitions involved in the lemma).

\begin{lemma}
\label{separate}
For each dimension $d\ge 2$ and positive integer $K$ there exists a constant 
$M(d,K)$ depending on $d$ and $K$ such that for any Coxeter diagram $\Sigma\in\mathcal{C}_{d,K}$ there exists an elliptic subdiagram $S\subset \Sigma$ of order $d$ and a node $x$ joined with $S$ by at most one dotted edge 
satisfying
$$
\det(\l S,x \r)\le-M(d,K)<0. 
$$

Moreover, if $\Sigma\in\mathcal{C}_2$ or $\mathcal{C}_3$,
then $M(d,K)=M(d)$ does not depend on $K$, and the subdiagram $\l S,x\r$ may be chosen in 
such a way that  for some $v\in S$ the subdiagram $\l S,x \r\setminus v$ contains no dotted 
edges and $S\setminus v$ contains no edges at all.
\end{lemma}

\begin{proof}
We prove the first statement of the lemma separately for $2$-dimensional case, 
$3$-dimensional case, and the case of bounded multiplicity. 
Let $P$ be a polytope with $\Sigma=\Sigma(P)$.

\medskip
\noindent
{\bf Case 1:} $\Sigma\in \mathcal{C}_2$.
We consider 3 cases: $P$ is either a triangle, or a right-angled polygon, or a polygon with more than $3$ sides and at least one strictly acute angle.

\smallskip

If $P$ is a Coxeter triangle, then $\det(\l S,x \r)$
is a decreasing function on the labels of all edges of $\Sigma$ (see Lemma~\ref{monoton}),
and it is easy to check that  
 $$\det(\l S,x \r)\le\det(L_{2,3,7})\approx -0.329$$

\smallskip

Suppose now that $P$ has at least $4$ sides.
Suppose that $P$ is not right-angled, and $a_1,a_2,a_3$ and $a_4$ are its consecutive sides such that $\angle a_1a_2<\pi/2$.
Denote by $v_1,\dots,v_4$ the nodes of $\Sigma$ corresponding to  $a_1,\dots,a_4$ respectively.
We will prove that either $\det(\l v_1, v_2,v_3\r)\le -1/2$ or $\det(\l v_2, v_3,v_4\r)\le -1/2$.

Since $P$ has at least $4$ sides, the line containing $a_1$ does not intersect the line containing $a_3$. 
Let $h_1$ be a unique common perpendicular to $a_1$ and $a_3$ (it is easy to see that $h$ intersects the sides $a_1$ and $a_3$ themselves but not the prolongations). 
Since $\angle a_1a_2<\pi/2$, $h_1$ does not coincide with $a_2$.
Let $h_2$  be a common perpendicular to $h_1$ and $a_2$ (it may coincide with $a_3$).
Then the lines $a_1,a_2,h_2$ and $h_1$ bound a 3 quadrilateral (i.e. a quadrilateral with $3$ right angles), see Fig.~\ref{lambert}.a. 
Denote by $\rho_1$ and $\rho_2$ the lengths of its sides contained in $h_1$ and $h_2$ respectively.
Then (see e.g.~\cite[Theorem~7.17.1]{B})
$$
\sinh \rho_1\, \sinh \rho_2 = \cos (\angle a_1a_2)\ge \frac{1}{2}$$
Therefore, either $\sinh \rho_1 \ge 1/\sqrt 2$ or  $\sinh \rho_2 \ge 1/\sqrt 2$.
Notice that the distance $\rho(a_1,a_3)$ between $a_1$ and $a_3$ is not less than $\rho_1$,
and the distance  $\rho(a_2,a_4)$ between $a_2$ and $a_4$ is not less than $\rho_2$.
So, either $\cosh^2 \rho(a_1,a_3)\ge 3/2$ or  $\cosh^2 \rho(a_2,a_4)\ge 3/2$.
In the former case
\begin{multline*}
$$\det(\l v_1, v_2,v_3\r)
=1-\cos^2(\angle a_1a_2)-\cos^2(\angle a_2a_3)-\cosh^2 \rho(a_1,a_3)-\\
-2\cos(\angle a_1a_2)\cos(\angle a_2a_3) \cosh \rho(a_1,a_3)\le   
1-\cosh^2 \rho(a_1,a_3)\le -1/2,$$
\end{multline*}
similarly, in the latter case,
$$\det(\l v_2, v_3,v_4\r)\le 1-\cosh^2 \rho(a_2,a_4)\le -1/2$$

\begin{figure}[!h]
\begin{center}
\psfrag{a}{\small a)}
\psfrag{b}{\small b)}
\psfrag{A1}{\footnotesize $A_1$}
\psfrag{1}{\footnotesize $a_1$}
\psfrag{2}{\footnotesize $a_2$}
\psfrag{3}{\footnotesize $a_3$}
\psfrag{n}{\footnotesize $a_n$}
\psfrag{n1}{\footnotesize $a_{n-1}$}
\psfrag{n2}{\footnotesize $a_{n-2}$}
\psfrag{h1}{\footnotesize $h_1$}
\psfrag{h2}{\footnotesize $h_2$}
\psfrag{fi}{\scriptsize $\angle a_1a_2$}
\psfrag{r1}{\footnotesize $\rho_1$}
\psfrag{r2}{\footnotesize $\rho_2$}
\epsfig{file=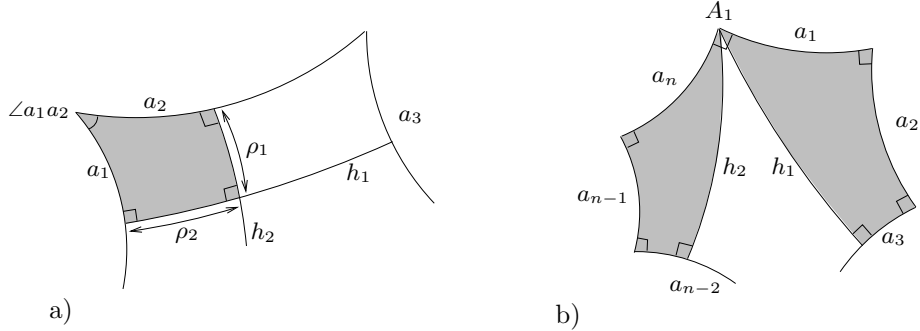,width=0.945\linewidth}
\caption{Lambert quadrilateral}
\label{lambert}
\end{center}
\end{figure}

Suppose now that $P$ is a right-angled polygon.
Let $a_1,a_2,a_3,\dots,a_n$ be its consecutive sides and let $A=a_n\cap a_1$ (see Fig.~\ref{lambert}.b).
Let $h_1$ be a line through $A$ orthogonal to $a_3$ and 
 $h_2$ be a line through $A$ orthogonal to $a_{n-2}$ (notice that since $P$ is right-angled $n\ge 5$, so $n-2\ge 3$). Then the lines $a_1,a_2,a_3$ and $h_1$ bound a Lambert quadrilateral,
as well as the lines $a_n,a_{n-1},a_{n_2}$ and $h_2$ do.
Furthermore, at least one of the angles $\angle a_1h_1$ and   $\angle a_nh_2$ does not exceed $\pi/4$ 
(since $\angle a_1a_n=\pi/2$).
Without loss of generality we may assume that  $\angle a_1h_1\le \pi/4$.
Then 
$$\sinh \rho(a_1,a_3) \sinh \rho(a_2,a_4)\ge \cos (a_1h_1)\ge 1/\sqrt{2}.$$
As it is shown above for the case of non- right-angled polygon,
this implies that either  
 $\det(\l v_1, v_2, v_3\r)\le -1/2$ or  $\det(\l v_2, v_3, v_4\r)\le -1/2$.

\medskip
\noindent
{\bf Case 2:} $\Sigma\in \mathcal{C}_3$.
First suppose that $P$ is a right-angled $3$-polytope.
Then all $2$-faces of $P$ are right-angled polygons, and as it is shown in Case~1, each of them has at least one edge
with length $\rho$ satisfying $\cosh \rho\ge \sqrt{3/2}$.
Let $e$ be such a long edge, $f_1$ and $f_2$ be the facets of $P$ such that $e=f_1\cap f_2$, and $f_3$ and $f_4$ be the remaining facets containing the vertices of $e$. Clearly, we have $\cosh \rho(f_1,f_2)\ge \sqrt{3/2}$ for the distance between $f_1$ and $f_2$.
Denote by $v_1,\dots,v_4$ the nodes of $\Sigma$ corresponding to $f_1,f_2,f_3$  and $f_4$ respectively. Then
$$\det(\l v_1,v_2,v_3,v_4\r)=1-\cosh^2\rho(f_1,f_2)\le -1/2$$
Suppose now that $P$ is not right-angled. Let $f_1$ and $f_2$ be the facets composing an acute dihedral angle, denote $e'=f_1\cap f_2$ the corresponding edge, and let $V$ be any vertex of $e'$. Denote by $f_3$ a unique facets of $P$ such that $V=f_1\cap f_2\cap f_3$. Then at least one of the angles formed by $f_3$ with $f_1$ and $f_2$ is right. Assume that $\angle(f_1f_3)=\pi/2$, and denote $e=f_1\cap f_3$. Let $V'$ be the other vertex of $e$, and $f_4$ be the facet such that $V=f_1\cap f_3\cap f_4$.  
If $f_2\cap f_4\ne \emptyset $ then $P$ is one of the nine $3$-dimensional hyperbolic Coxeter tetrahedra, so $\det(\l v_1,v_2,v_3,v_4\r)\le -M_0$ for some 
constant $M_0>0$. Suppose that $f_2$ does not intersect $f_4$.
Then we have $a_{13}=0$, $a_{24}=-\cosh \rho(f_3,f_4)< -1$,
and $a_{12}$ is not equal to $0$, i.e. is not greater than $-1/2$.
Expanding the determinant
$\det(a_{ij})$ and taking into account that $a_{13}=0$ and $a_{ij}\le 0$ for $i\ne j$, we get
\begin{multline*}
$$\det(\l v_1,v_2,v_3,v_4\r)=1-a_{12}^2-a_{14}^2-a_{23}^2-a_{24}^2-a_{34}^2+\\
+2a_{12}a_{24}a_{14}+2a_{23}a_{34}a_{24}-2a_{12}a_{23}a_{34}a_{14}-a_{12}^2a_{34}^2-a_{14}^2a_{23}^2=\\
=(1-a_{24}^2)-(a_{12}-a_{34})^2-2a_{12}a_{34}-a_{14}^2-a_{23}^2-a_{12}^2a_{34}^2-a_{14}^2a_{23}^2+\\
+2a_{24}(a_{12}a_{14}+a_{23}a_{34})-2a_{12}a_{23}a_{34}a_{14}\le -(a_{12}-a_{34})^2-2a_{12}a_{34}$$
\end{multline*}
If $a_{34}\ne 0$, then $a_{34}\le -1/2$ and $\det(a_{ij})\le -1/2$.
If $a_{34}= 0$, then $a_{12}-a_{34}\le -1/2$ and $\det(a_{ij})\le -1/4$.

\medskip
\noindent
{\bf Case 3:}  $\Sigma\in \mathcal{C}_{d,K},\ d\ge 4$.

First suppose that $P$ is not right-angled, so $\Sigma$ contains at least one non-dotted edge, say $\l v_1,v_2 \r$. Then there exists an elliptic subdiagram $S \subset \Sigma$ containing $\l v_1,v_2 \r$.
A subdiagram $S\setminus v_1$ defines an edge ($1$-face) of $P$, so it has one more 
good neighbor $x$ in $\Sigma$ besides $v_1$. By Lemma~\ref{monoton}, $\det(\l S,x \r)$ 
is a decreasing function on label of each edge joining $x$ with $S$. Notice that $(x,v_1)$ is the only pair of nodes in $\l S,x \r$ which can be joined by a dotted edge.
If it is, replace $\l x,v_1\r$ by the edge labeled by $7$ and denote the obtained diagram by $U$, otherwise let $U=\l S,x\r$. 
Since $\l v_1,v_2\r$ is a non-empty edge, the diagram $U$ is also hyperbolic.
So, $\det(U)<0$ and  $\det(\l S,x\r)\le\det(U)$.
Since $\Sigma\in \mathcal{C}_{d,K}$, there are finitely many possibilities for $U$,
so we may take as $M(d,K)$ any positive number smaller than minimum of $|\det(U)|$.

Now suppose that $P$ is right-angled. Any $2$-face of $P$ is right-angled polygon, so, by Case~1, it has an edge $VV'$ of length $\rho\ge \arcsinh{(1/\sqrt{2})}$. Denote by $f_1,\dots,f_d$ the facets containing $V$, we may assume that $V'$ is the intersection of facets $f_0,f_1,\dots,f_{d-1}$. Let $v_0,\dots,v_d$ be the corresponding nodes of $\Sigma$. Then the diagram $\l v_0,\dots,v_d\r$ contains a unique edge $\l v_0,v_d\r$, and $\det(\l v_0,\dots,v_d\r)\le -1/2$. 

\medskip

So, the first statement of the lemma is proved. The second statement follows immediately from the choice of $\l S,x\r$ and the constant $M(d)$ in low-dimensional cases. 
\end{proof}

\noindent
The following lemma provides an upper bound for a local determinant of hyperbolic diagram of special type. Later on, we find such a subdiagram in any Coxeter diagram of a polytope from ${\mathcal P}$, and use Lemma~\ref{local} to get a bound for a multiplicity of an edge.

\begin{lemma}
\label{local}
Given $M>0$, integer $k\ge 5$ and dimension $d\ge 2$,
there exists a constant $C(d,K,M)>0$   
such that for any diagram $\l S,x,y\r$ satisfying the following five conditions: 
\begin{itemize}
\item[1)]
$S$ is an elliptic subdiagram of order $d$, while  $\l S,x\r$ and  $\l S,x,y\r$ 
are hyperbolic diagrams,
\item[2)]
$x$ is joined with $S$ by at most one dotted edge,
\item[3)]
if $v\in S $ and $\l v,x\r$ is a dotted edge, then  labels of edges of $S\setminus v$ do not exceed $K$,
\item[4)]
$y$ is not joined with  $\l S,x\r$ by any dotted edge,
\item[5)]
 $\det(\l S,x \r)<-M<0$,
\end{itemize}
the inequality
$$ 0\le\det(\l S,x,y\r,y)\le C(d,K,M)$$
holds.
\end{lemma} 

\begin{proof}
The inequality  $ 0\le\det(\l S,x,y\r,y)$ follows from the fact that both $\l S,x,y \r$ and $\l S,x \r$
are hyperbolic, so their determinants are non-positive.

Suppose that the diagram  $\l S,x,y\r$ contains no dotted edges.
Then we have $|\det(\l S,x,y\r)|\le (d+2)!$ since each of the $(d+2)!$ summands in the standard expansion cannot exceed $1$.
So, $ 0<\det(\l S,x,y\r,y)\le \frac{(d+2)!}{M}$.

Now, suppose that  $\l S,x,y\r$ contains a unique dotted edge $\l x,v\r$ and denote by $\rho$ the label of this edge.
Then $$\det(\l S,x,y\r,y)=\frac{a_1 \rho^2 + b_1 \rho +c_1}{a_2 \rho^2 + b_2 \rho +c_2},$$
where $|a_i|,|b_i|,|c_i|\le (d+2)!$ Moreover, $a_2=-\det(\l S\setminus v\r)$, at the same time
 $\l S\setminus v\r$ is an elliptic diagram containing no
edges labeled by $k\ge K$, so we obtain $|a_2|\ge N(d,K)>0$ for some constant $N(d,K)$. We may assume $N(d,K)$ to be less than $1$.

Therefore, for $\rho\ge 4\frac{(d+2)!}{N(d,K)}$ we have
\begin{multline*}
\det(\l S,x,y\r,y)=\frac{\left|\frac{a_1}{a_2}+\frac{b_1}{a_2\rho}+\frac{c_1}{a_2\rho^2}\right|}{\left|1+\frac{b_2}{a_2\rho}+\frac{c_3}{a_2\rho^2}\right|}\le
\frac{\left|\frac{a_1}{a_2}\right|+\left|\frac{b_1}{a_2\rho}\right|+\left|\frac{c_1}{a_2\rho^2}\right|}{1-\left|\frac{b_2}{a_2\rho}\right|-\left|\frac{c_3}{a_2\rho^2}\right|}\le\\
\le \frac{\left|\frac{a_1}{a_2}\right|+\frac{1}{4}+\frac{N(d,K)}{8(d+2)!}}{1-\frac{1}{4}-\frac{N(d,K)}{8(d+2)!}}<
\frac{|a_1/a_2|+1/2}{1/2}\le\frac{2(d+2)!}{N(d,K)}+1
\end{multline*}
For $\rho< 4\frac{(d+2)!}{N(d,K)}$ 
we have  $$ 0<\det(\l S,x,y\r,y)\le \frac{(d+2)!+\frac{4(d+2)!^2}{N(d,K)}+\frac{16(d+2)!^3}{N(d,K)^2}}{M},$$
so $\det(\l S,x,y\r,y)$ does not exceed maximum of these two numbers for any $\rho$.
\end{proof}


%
%
%
%
%

\noindent
At the end of the section, we prove several elementary facts about Coxeter diagrams of polytopes belonging to ${\mathcal P}$. 

\begin{lemma}
\label{dplus2}
Let $P\in {\mathcal P_{(d,n)}}$, and let $\Sigma$ be a Coxeter diagram of $P$. Then for any elliptic subdiagram $S\subset\Sigma$ of order $d$ there exist nodes $x,y\in \Sigma\setminus S$ such that the diagram $\l S,x,y\r$ contains no dotted edges.
\end{lemma}

\begin{proof}
Denote by $\mathcal M$ the set of all nodes of $\Sigma\setminus S$ which are not joined with $S$ by a dotted edge.
Denote by $n_0$ the number of nodes in $\mathcal M$, and by $n_1=n-d-n_0$ the number of the remaining nodes in $\Sigma\setminus S$.
Clearly, $$n_0\ge (n-d)-p\ge (n-d)-(n-d-2)=2$$
Now suppose that any two elements of $\mathcal M$ are joined by a dotted edge. 
Then, using that $n_0\ge 2$, we may write down the following inequality for $p$:
$$p\ge \frac{n_0(n_0-1)}{2}+n_1=\frac{n_0(n_0-1)}{2}+n-d-n_0=n-d+\frac{n_0(n_0-3)}{2}\ge n-d-1$$
The contradiction shows that there are at least two nodes $x,y\in\mathcal M$ which are not joined by a dotted edge. Therefore, 
$\l S,x,y\r$ contains no dotted edges.
\end{proof}

\noindent
The same reasoning as above gives rise to the following statement.

\begin{lemma}
\label{dplus11}
Let $P\in {\mathcal P_{(d,n)}}$, and let $\Sigma$ be a Coxeter diagram of $P$. Then for any elliptic subdiagram $S\subset\Sigma$ of order $d$ and any $x\in\Sigma$ such that 
 $\l S,x\r$ contains no dotted edges there exists a node $y\in \Sigma\setminus \l S,x\r$ such that the diagram $\l S,x,y\r$ contains no dotted edges.
\end{lemma}

\begin{lemma}
\label{dots}
Suppose that $P\in {\mathcal P_{(d,n)}}$ 
and $S_0\subset \Sigma(P)$ is an elliptic diagram of the type $G_2^{(k)}$, $k\ge 6$.
If each neighbor of $S_0$ is joined with $\o S_0$ by at least one dotted edge then 
\begin{itemize}
\item[1)] $P(S_0)\in {\mathcal P}_{(d',n')} $ 
(where $d'=d-2$ and $n'$  are the dimension of $P(S_0)$ 
and the number of nodes  in $\Sigma_{S_0}$ respectively);

\item[2)] For any elliptic subdiagram $S\subset \o S_0$ of order $d-2$ there exists a node $x\in \o S_0\setminus S$
which is not joined with $S$ by a dotted edge. Moreover, there exists a neighbor $y$ of $S_0$ such that the diagram $\l S,x,y\r$ contains no dotted edges.
\end{itemize}
\end{lemma}

\begin{proof}
Denote by $b$ the number of neighbors of $S_0$. Denote by $p'$ the number of dotted edges in $\Sigma_{S_0}$.
 By assumption, each neighbor of $S_0$ is joined with $\o S_0$ by at least one dotted edge, so there are at least $b$ dotted edges ``outside'' of $\o S_0$. Taking into account that $p\le n- (d+2)$, we have $$p'\le p-b\le n-(d+2)-b=n-d'-4-b$$
On the other hand, $n'=n-|S_0|-b=n-2-b$, so $n=n'+b+2$. Thus, we get $$p'\le n'+b+2-d'-4-b=n'-(d'+2),$$ and the first statement is proved.

The proof of the second statement is similar to the proof of Lemma~\ref{dplus2}. 
Denote by $\mathcal M$ the set of all nodes of $\o {S_0}\setminus S$ which are not joined with $S$ by a dotted edge.
Denote $n_0=|\mathcal M|$, and denote by $n_1=n'-d'-n_0$ the number of the remaining nodes in $\o {S_0}\setminus S$.
By Lemma~\ref{dplus2}, $n_0\ge 2$.
Now suppose that each neighbor of $S_0$ is joined with each node of $\mathcal M$ by a dotted edge.
Then there are $bn_0$ dotted edges joining the neighbors of $S_0$ with the nodes of $\mathcal M$.
Therefore, $$p\ge bn_0+n_1\ge b+n_0-1+n_1=n-d-1>n-d-2$$ in contradiction to the assumption. 
Thus, there exist a neighbor $y$ of $S_0$ and $x\in{\mathcal M}$ which are not joined by a dotted edge, so the lemma is proved.  
\end{proof}





\section{Upper bound for multiplicity of edge}
\label{multiplicity}
In this section we prove that for each $d$ there is a lower bound for dihedral angles of polytopes from $\mathcal{P}_d$.
We start from low-dimensional case.

\begin{lemma}
\label{no_multi}
Let $P$ be a Coxeter $d$-polytope, and let
 $S_0\subset \Sigma(P)$ be a subdiagram of the type $G_2^{(k)}$, $k\ge 6$.
Suppose that either $d=4$, or $d=5$, or 
 $\o S_0$ contains no multi-multiple edges.

%

Then there exists a constant $K_1(d)$ depending on $d$ only such that $\Sigma(P)\in\mathcal{C}_{d,K_1(d)}$ (i.e., $k\le K_1(d)$).
\end{lemma}

\begin{proof}
Consider the diagram $\o S_0$. By assumptions, we may apply Lemma~\ref{separate} with $K=5$ to $P(S_0)$. Thus, there exists an elliptic subdiagram $S\subset \o S_0$ of order $d-2$ and a node $x \in \o S_0$ such that $\det(\l S,x\r)<-M_0<0$ for some constant $M_0>0$ depending only on $d-2$, and $\l S,x\r$ contains at most one dotted edge.
%

Suppose that there exists a neighbor $y$ of $S_0$ such that $y$ is not joined with $\l S,x\r$ 
by a dotted edge. Then the diagram $\l S,x,y\r$ fits into all the conditions of Lemma~\ref{local} (with $K=5$ and $M=M_0$), so there exists a constant $C$
(depending on $d$ only) such that $0<\det((\l S,x,y \r,y)< C$. 
By Proposition~\ref{loc_sum}, $$ \det(\l S_0,S,x,y \r,y)=\det(\l S_0,y\r,y)+\det(\l S,x,y \r,y)-1.$$
Since the diagram $\l S_0,S,x,y \r$ contains $d+2$ nodes, $\det(\l S_0,S,x,y \r)=0$ 
and hence, $ \det(\l S_0,S,x,y \r,y)=0$ and  
$$\det(\l S_0,y\r,y)=1-\det(\l S,x,y \r,y)>1-C.$$
Since $\det(\l S_0,y\r,y)$ tends to (negative) infinity
while $k$ tends to infinity, we obtain some constant $K_1(d)$ such that  $k\le K_1(d)$.

Suppose now that each neighbor of $S_0$ is joined with  $\l S,x\r$  by a dotted edge.
Then we are in assumptions of Lemma~\ref{dots},
so there exists a node $x'\in \mathcal \o S_0\setminus S$ and a neighbor $y'$ of $S_0$ such that $\l S,x',y'\r$
contains no dotted edges. Now we can apply Lemma~\ref{separate1} to $\l S,x'\r$, which implies that there exists a constant $M_1>0$ depending on $d$ only such that $\det(\l S,x' \r)< -M_1$. Therefore, $\l S,x',y'\r$ fits into assumptions of Lemma~\ref{local} with $K=5$ and $M=M_1$, so we can proceed with the diagram  $\l S,x',y'\r$ in the same way as we did in the preceding paragraph with  $\l S,x,y\r$. This finishes the proof.
\end{proof}

\noindent
Now we are able to prove the main result of the section. The following theorem provides a uniform upper bound for multiplicity of edges of Coxeter diagrams $\Sigma(P)$ for $P\in\mathcal{P}_d$.

\begin{theorem}
\label{multi}
If $d\ge 4$,
then there exists a constant $K_0(d)$ such that  
for any $P\in\mathcal{P}_d$ the Coxeter diagram $\Sigma(P)$ belongs to $\mathcal{C}_{d,K_0(d)}$.
\end{theorem}

\begin{proof}
The proof is by induction on $d$. The  base consists of the cases $d=4$ and $d=5$,
for which we refer to Lemma~\ref{no_multi}.

Suppose that $d\ge 6$ and for each $4\le d'<d$ the theorem is already proved, i.e.   
for any $d'$ there is a constant $K_0(d')$ such that for any $P'\in\mathcal{P}_{d'}$ and for any $S_0\subset\Sigma(P')$ of type $G_2^{(k)}$ we have $k\le K_0(d')$.

Now take any subdiagram $S_0\subset \Sigma(P)$ of type $G_2^{(k)}$ and consider the diagram $\o S_0$. 
We may assume that $k\ge 6$, in particular $\Sigma_{S_0}=\o S_0$.
If $\o S_0$ contains no multi-multiple edges, we refer to Lemma~\ref{no_multi}.
Suppose that $\o S_0$ contains at least one multi-multiple edge. Then choose any elliptic subdiagram $S\subset \o S_0$ of order $d-2$
containing a multi-multiple edge.
We can write $S$ as $\l S_1,\dots,S_r,R' \r$, $r\ge 1$, where each of $S_1,\dots,S_r$ is a subdiagram of order $2$ consisting of a multi-multiple
edge, and $R'$ is either a subdiagram containing no multi-multiple edges or a subdiagram of order $2$ or $3$ (such a decomposition of $S$ may not be unique). 

Denote $T_0=\l S_0,S_1,\dots,S_r\r$, and consider $T= \o T_0\subset\o S_0$ (so, $T$ consists of all nodes of $\o S_0$ not joined with 
$\l S_0,S_1,\dots,S_r\r$). This is a Coxeter diagram of a $(d-2r-2)$-polytope, where $d-2r-2\ge 2$. We may assume that either $d-2r-2\le 3$
or $T$ contains no multi-multiple edges, otherwise we may add a multi-multiple edge (denote it by $S_{r+1}$) to $T_0$ to get $T_0'$, and consider $T'=\o {T_0'}$.  
By Lemma~\ref{separate}, there exists a constant $M_0>0$ (depending on $d$ only), an elliptic subdiagram $R\subset T$ of order $(d-2r-2)$ and a node $x\in T\setminus R$ 
such that $\det(\l x,R\r)<-M_0$.  
  
First, suppose that there exists a neighbor $y$ of $S_0$ not joined with $\l x,R\r$ by a dotted edge. 
Then, applying Proposition~\ref{loc_sum}, we get  $$\det(\l S,x,y \r,y)\!=\!\det(\l S_1,y \r,y)+\cdots+\det(\l S_r,y \r,y)+\det(\l R,x,y \r,y)-r$$
Notice that $\det(\l S,x,y \r,y)\ge 0$ while each of $\det(\l S_i,y \r,y)$ either equals to 1 (if $S_i$ is not joined with $y$) or is negative (otherwise). Therefore, $\det(\l S_i,y \r,y)-1\le 0$ for $i=1,\dots,r$, and 
we obtain $$0\le\det(\l S,x,y \r,y)\le \det(\l R,x,y \r,y).$$
At the same time, by Lemma~\ref{local} applied to $\l R,x,y \r$ with $K=5$ and $M=M_0$, the local determinant $\det(\l R,x,y \r,y)$ is bounded by some constant $C$ (depending on $d$ only). 

Now consider the diagram $\l S_0,S,x,y \r$. By construction, $\l S,x\r\subset\o S_0$, so $\l S_0,S,x\r$ is not connected, and we use Proposition~\ref{loc_sum} to obtain 
 $$\det(\l S_0,S,x,y \r,y)=\det(\l S_0,y\r,y)+\det(\l S,x,y\r,y)-1.$$
Since the diagram $\l S_0,S,x,y \r$ consists of $d+2$ nodes, we get $$0=\det(\l S_0,S,x,y \r)=\det(\l S_0,S,x,y \r,y),$$
which implies 
$$\det(\l S_0,y\r,y)=1-\det(\l S,x,y \r,y)>1-C.$$
Since $\det(\l S_0,y\r,y)$ tends to (negative) infinity
while $k$ tends to infinity, we obtain some constant $K_2(d)$ such that  $k\le K_2(d)$.

We are left to consider the case when each neighbor of $S_0$ is joined with   $\l R,x\r$ by a dotted edge,
suppose now that this holds.
Let $S$ be any elliptic subdiagram of $\o S_0$. According to the second statement of Lemma~\ref{dots}, there exists a node $x'\subset \o S_0$ and a neighbor $y'$ of $ S_0$ such that the diagram $\l S,x',y'\r$ contains no dotted edges. 
By the first statement of Lemma~\ref{dots}, we have $P(S_0)\in\mathcal{P}_{d-2}$. We may apply the induction assumption to get $\Sigma_{S_0}\in\mathcal{C}_{d-2,K_0(d-2)}$. In particular, the subdiagram
 $\l S,x'\r$ contains no edges labeled by $k'>K_0(d-2)$.
So, by Lemma~\ref{separate}, $-M_1 \le\det \l S,x'\r<0$ for some $M_1>0$ depending on $d-2$ only, and by Lemma~\ref{local}
there exists a constant $C_1>0$ such that  $0\le\det(\l y,x,S \r,y)\le C_1$.
This implies that $$\det(\l S_0,y\r,y)=1-\det(\l S,x,y \r,y)>1-C_1$$
and hence, we obtain some constant $K_3(d)$ such that  $k\le K_3(d)$. 

Now, taking as $K_0(d)$ the maximum of $K_2(d)$ and $K_3(d)$, we complete the proof.  
\end{proof}

\section{Finiteness of the number of polytopes}
\label{fin}

In the previous section we proved that multiplicity of edges of Coxeter diagrams of polytopes from $\mathcal{P}_d$ is bounded. The current goal is to prove that the number of facets of polytopes from $\mathcal{P}_d$ is also bounded by some constant depending on $d$ only. This will imply finiteness of $\mathcal{P}_d$.  

First, we estimate the order of subdiagrams without dotted edges.

\begin{lemma}
\label{n_0}
There exists a constant $Q_0(d,k)$ such that for any diagram $\Sigma'$ satisfying 
\begin{itemize}
\item[1)]
$\Sigma'\subset \Sigma$ for some diagram $\Sigma\in  \mathcal{C}_{d,k}$,
\item[2)]
$\Sigma'$ contains an elliptic subdiagram $S$ of order $d$,
\item[3)]
$\Sigma'$  contains no dotted edges,
\end{itemize}
the inequality
$|\Sigma'|\le Q_0(d,k)$ holds.

\end{lemma}

\begin{proof}
Assume that $|\Sigma'|>d$. Then there exists $x\in\Sigma'\setminus S$.
Since $\Sigma'\subset \Sigma\in  \mathcal{C}_{d,k}$, multiplicity of each edge of $\l S,x\r$ is uniformly bounded, so the number of possible configurations $\l S,x\r$ is finite, denote the number of them by $N_0(d,k)$. 

Recall that the signature of $\l S,x\r$ is $(d,1)$. This implies that the corresponding vectors in $\R^{(d,1)}$ (denote them $v_0,\dots,v_d$) form a basis, so for any $y\in\Sigma'\setminus\l S,x\r$ the corresponding vector $v\in \R^{(d,1)}$ is completely determined by the weights of the edges joining $y$ with  $\l S,x\r$.
Again, since $\Sigma'\subset \Sigma\in  \mathcal{C}_{d,k}$, for each $\l S,x\r$ there is at most $(k-1)^{d+1}$ ways to join $y$ with $\l S,x\r$. Thus, for each of finite number of configurations of vectors ($v_0,\dots,v_d$) there are finitely many vectors $v$ which may correspond to nodes of $\Sigma'\setminus\l S,x\r$, more precisely the number is bounded by $N_0(d,k)(k-1)^{d+1}$. Adding to this number the order of $\l S,x\r$, we get the bound $Q_0(d,k)$.
\end{proof}

\begin{example}
To illustrate application of Lemma~\ref{n_0} we provide a very rough bound for $4$-dimensional case. Let $\Sigma(P)\in\mathcal{C}_{4,k}$. We will estimate the number of facets of $P$.

First, compute the number of distinct elliptic diagrams $S$ of order $4$. There are $5$ connected ones ($A_4$, $B_4$, $D_4$, $F_4$, and $H_4$), $3$ non-connected ones with connected component of order $3$ (direct sum of $A_1$ and one of $A_3$, $B_3$, $H_3$), and $k(k-1)/2$ diagrams of type $G_2^{(m)}+G_2^{(l)}$, $2\le l\le m\le k$. So, the number of elliptic diagrams is $8+k(k-1)/2$.

Next, compute $N_0(4,k)$. A node $x\in\Sigma(P)\setminus S$ may be attached to an elliptic diagram $S$ in $(k-1)^4$ ways (this bound is very rough: we forget about huge number of symmetries; moreover, a large part of these configurations gives rise to elliptic or parabolic diagrams). So, $$N_0(4,k)\le (8+k(k-1)/2)(k-1)^4$$.

Now, any node $y\in\Sigma(P)\setminus\l S,x\r$ may be attached to $\Sigma(P)\setminus\l S,x\r$ in at most $(k-1)^5$ ways (again, this bound is very rough). Hence, we may estimate that $|\Sigma(P)|\le Q_0(4,k)$, where
$$Q_0(4,k)\le N_0(4,k)(k-1)^5+5\le (8+k(k-1)/2)(k-1)^9+5$$   
\end{example}

\begin{lemma}
\label{n_1}
There exists a constant $Q_1(d,k)$ such that for any diagram $\Sigma'$ satisfying 
\begin{itemize}
\item[1)]
$\Sigma'\subset \Sigma$ for some diagram $\Sigma\in  \mathcal{C}_{d,k}$,
\item[2)]
$\Sigma'$ contains an elliptic subdiagram $S$ of order $d$ 
and a node $x\notin S$ such that 
\begin{itemize}
\item[2a)]
$\l S,x\r$ contains no dotted edges,
\item[2b)]
at most one node of $\Sigma'\setminus \l S,x\r$ is joined with $\l S,x\r$ by a dotted edge,
\end{itemize}
\end{itemize}
the inequality
$|\Sigma'|<Q_1(d,k)$ holds.
\end{lemma}

\begin{remark}
\label{lin}
Before proving Lemma~\ref{n_1} we recall some elementary fact from linear algebra. Let $(V,g)$ be a non-degenerate quadratic space, and let $(v_0,\dots,v_d)$ be a basis. Then a system of equations 
$$(v,v)=1,\ (v,v_i)=c_i, 1\le i\le d$$
has at most two solutions. In other words, any vector is determined (up to some symmetry) by its scalar products with all but one elements of basis of $V$.
\end{remark}

\begin{proof}[of Lemma~\ref{n_1}]
We may assume that there exist $y\in\Sigma'\setminus\l S,x\r$ and $z\in \l S,x\r$ joined by a dotted edge, otherwise we take $Q_1(d,k)=Q_0(d,k)$ (see Lemma~\ref{n_0}).

Now the proof is similar to the proof of Lemma~\ref{n_0}. There are finitely many diagrams $\l S,x\r$, and  
finitely many ways to join $y$ with $\l S,x\r\setminus z$. In view of Remark~\ref{lin}, 
the weight of the edge $\l y,z\r$ is determined by the multiplicities of the edges joining $y$ with  $\l S,x\r\setminus z$. 
Thus, we get some bound for the number of nodes in $\Sigma'$. 

More precisely, we may take a rough bound
$$Q_1(d,k)=2(d+1)N_0(d,k)(k-1)^d$$
where the multiple $d+1=$${d+1}\choose{d}$ is the number of ways to choose $z$ in $\l S,x\r$
and $N_0(d,k)$ is the estimate for the number of possible configurations   $\l S,x\r$ (see proof of Lemma~\ref{n_0}). 
\end{proof}




\begin{lemma}
\label{n}
For any $d\ge 4$ there exists a constant $n_0(d)$ such that 
any polytope from  ${\mathcal P}_d$ has at most $n_0(d)$ facets.
\end{lemma}

\begin{proof}
Let $P\in {\mathcal P}_d$ and $\Sigma=\Sigma(P)$, i.e. $p\le n-d-2$, where $n=|\Sigma|$ and $p$ is the number of dotted edges in $\Sigma$.
By Lemma~\ref{multi}, $\Sigma\in \mathcal{C}_{d,k}$ for $k=K_0(d)$ in notations of Lemma~\ref{multi}.

Let $S$ be any elliptic subdiagram of $\Sigma$ of order $d$. By Lemma~\ref{dplus2}, there exist $x,y\in\Sigma\setminus S$ such that the diagram $\l S,x,y\r$ contains no dotted edges.
Let $\mathcal {M}_{\le 1}$ be the set of nodes of $\Sigma\setminus\l S,x,y\r$ joined with $\l S,x,y\r$ by at most one dotted edge.
By Lemma~\ref{n_1}, $|\mathcal {M}_{\le 1}|\le Q_1(d,K_0(d))$.
Denote by  $\mathcal {M}_{>1}$ the set of nodes of $\Sigma$ joined with  $\l S,x\r$  by two or more dotted edges.
Clearly,  $|\mathcal {M}_{>1}|\le p/2$. Furthermore,
$$
n=(d+2)+|\mathcal {M}_{\le 1}|+|\mathcal {M}_{>1}|\le (d+2)+Q_1(d,K_0(d))+p/2. 
$$
Since  $p\le n-d-2$, we have
$$n\le Q_1(d,K_0(d))+n/2+d/2+1$$,
which implies 
$$n\le n_0(d)=2Q_1(d,K_0(d))+d+2.$$
\end{proof}

\begin{cor}
\label{finite}
The number of polytopes in ${\mathcal P}_d$ is finite.
\end{cor}

\begin{proof}
Let $P\in {\mathcal P}_d$ and $\Sigma=\Sigma(P)$. By Lemma~\ref{n}, the number of nodes in $\Sigma$ is bounded by some constant $n_0(d)$ depending on $d$ only.
According to Lemma~\ref{multi}, $\Sigma\in \mathcal{C}_{d,K_0(d)}$. Clearly, the number of Coxeter diagrams with  bounded number of nodes and bounded multiplicities of edges is finite (we do not write weights of dotted edges).  
By Andreev's Theorem~\cite{An} each of the possible diagrams corresponds to at most one Coxeter polytope.
So, there are finitely many Coxeter diagrams that may occur to be diagrams of polytopes from ${\mathcal P}_d$, which implies that ${\mathcal P}_d$ is finite.
\end{proof}

Combining Corollary~\ref{finite} with the result of Vinberg~\cite{abs} stating that the dimension of compact hyperbolic Coxeter polytope does not exceed
$29$, we obtain the following theorem.

\begin{theorem}
\label{main}
The set ${\mathcal P}$ is finite.
\end{theorem}

\section{Algorithm}
\label{algorithm}

In Section~\ref{fin} we proved that ${\mathcal P}_d$ is finite.
Since we have a bound on the number of facets of any polytope $P\in{\mathcal P}_d$,
we are limited to finitely many combinatorial types of polytopes.
Given a combinatorial type, we have finitely many possibilities to assign weights 
to non-dotted edges of $\Sigma=\Sigma(P)$ (due to Lemma~\ref{multi}).
However, the possible labels of dotted edges remain undefined.
In particular, given a combinatorial type of $d$-polytope (satisfying the condition $p\le n-d-2$) we even can not still check 
whether there exists a compact Coxeter polytope of this combinatorial type.

In this section we describe an algorithm which allows us to determine the labels of the dotted edges one after another. 
Using the algorithm one can check if the given combinatorial type is realizable in ${\mathcal P}_d$, as well as to find all possible realizations in this class.

The main idea of the algorithm is the following. First, for any polytope $P\in{\mathcal P}_d$ and its Coxeter diagram $\Sigma(P)$ (with unknown labels of dotted edges) we show the way to find out the weights of dotted edges. For this we point out a sequence of subdiagrams 
$$\Sigma_0\subset\Sigma_1\subset\dots\subset\Sigma_{n-d}=\Sigma(P),\quad |\Sigma_i|=d+i$$
together with a way to determine weights of dotted edges in $\Sigma_i$ by the weights of dotted edges of  $\Sigma_{i-1}$. 

Next, we use results of the previous two sections to show that for any $i\le n_0(d)$ there are finitely many ways only to assign weights to $\Sigma_i$. Then, starting from all possible elliptic diagrams $\Sigma_0$ and adding vertices one by one, we get a (huge) list of diagrams which are candidates to be subdiagrams of $\Sigma(P)$ for $P\in{\mathcal P}_d$. At each step we check the signature of every diagram and preserve only ones with correct signature (which is $(d,1)$). For diagrams with correct one we also check whether it is already a diagram of a polytope. If at some step $i$ the set of diagrams $\Sigma_i$ is empty, the procedure is done. This cannot happen later than at step number $n_0(d)-d$. 

In this way we get Coxeter diagrams of all the polytopes from   ${\mathcal P}_d$.

\begin{lemma}
\label{alg_p}
Given a polytope $P\in{\mathcal P}_d$ and its Coxeter diagram $\Sigma=\Sigma(P)$ with unknown weights of dotted edges, there exists a finite algorithm which
provides the weights of dotted edges of $\Sigma$. 
\end{lemma}

\begin{proof}
Let $v_0\in\Sigma$ be a node incident to the maximal number of dotted edges (in particular, this means that any node of $\Sigma$ other than $v_0$ is incident to at most $p/2+1$ dotted edges). 
Let $S_0\subset \Sigma $ be any elliptic subdiagram of order $d$ containing $v_0$
(i.e., $S_0$ is a diagram of some vertex $V_{S_0}$ of $P$, and $V_{S_0}$ belongs to facet $F_{v_0}$ corresponding to $v_0$).
Define $\Sigma_0=S_0$.
As we have explained above, for each positive integer $i\le n-d$ we will construct a subdiagram 
$\Sigma_i\supset\Sigma_{i-1}$ of order $d+i$ such that in each $\Sigma_i$ all the weights are determined.

\medskip

\noindent
{\bf Step 1: attaching nodes.} 
By Lemma~\ref{dplus2}, there exist $x,y\in\Sigma\setminus S_0$ such that $\l S_0,x,y\r$ contains no dotted edges.
Define $\Sigma_1=\l S_0,x\r$, $\Sigma_2=\l S_0,x,y\r$.  

Denote by $\mathcal{M}_{\le 1}$ the set of nodes of $\Sigma\setminus\Sigma_2$ joined with $\Sigma_2$ by at most one dotted edge, and let $z\in\mathcal{M}_{\le 1}$.
According to Remark~\ref{lin}, we can find the label of the dotted edge (if any) in $\l S_0,x,z\r$,
so we can find the vector in  $\R^{(d,1)}$ corresponding to $z$ as a linear combination of the vectors corresponding to  $\l S_0,x\r$, 
which allows us to find all the labels of the edges joining $z$ with $\Sigma_2$. Hence, we can take $\Sigma_3=\l\Sigma_2,z\r$.

In the same way, we can attach to $\Sigma_3$ any other element of $\mathcal{M}_{\le 1}$ to get $\Sigma_4$. Indeed, we find the corresponding vector in $\R^{(d,1)}$ as a linear combination of the vectors corresponding to $\l S_0,x\r$, and then compute scalar products with all the vectors corresponding to the remaining nodes $y,z$ of $\Sigma_3$. Applying this procedure for each node from $\mathcal{M}_{\le 1}$, we get a subdiagram $\Sigma_m\subset\Sigma$, where $m=2+|\mathcal{M}_{\le 1}|$, 
$|\Sigma_m|=d+2+|\mathcal{M}_{\le 1}|$.  

Each of the remaining nodes is joined with  $\Sigma_1=\l S_0,x\r$ by at least two dotted edges.     
Since   $p\le n-d-2$, this implies that there are at most $(n-d-2)/2$ nodes in $\Sigma\setminus \Sigma_m$.
In other words, to this moment $\Sigma_m$ consists of $d+2$ nodes of $\l S_0,x,y\r$
and at least $n-d-2-(n-d-2)/2=(n-d)/2-1$ attached nodes, so $m\ge (n-d+2)/2$. If $m=n-d$ then the lemma is proved, otherwise we go to the next step.

\medskip

\noindent
{\bf Step 2: walking along edges of $P$.}
Let $m'\ge m$, and suppose that we have already constructed a subdiagram $\Sigma_{m'}$. 
Let $S\subset \Sigma_{m'}$ be any diagram of a vertex (i.e., an elliptic diagram of order $d$), and let $w\notin\Sigma_{m'}$ be a node such that $\l S,w\r$ is a complete diagram of an edge.

\smallskip

\noindent
{\bf Claim:} It is possible to attach $w$ to $\Sigma_{m'}$ to get $\Sigma_{m'+1}$.
\begin{proof}[of the claim]
Let $x'$ be any node of $\Sigma_{m'}\setminus S$. 
Since $\l  S,w\r$ is a complete diagram of an edge, $w$ can be joined with $S$ by at most one dotted edge.

Suppose that $w$ is not joined with $S$ by a dotted edge. Take any $x'\in\Sigma_{m'}\setminus S$ and notice that $\l S,x'\r$ has signature $(d,1)$, so it corresponds to a basis of $\R^{(d,1)}$. Therefore, we are able to find the corresponding vector in $\R^{(d,1)}$ as a linear combination of the vectors corresponding to $\l S,x'\r$, and then compute scalar products with all the vectors corresponding to the remaining nodes of $\Sigma_{m'}$, so $\Sigma_{m'+1}=\l\Sigma_{m'+1},w\r$.

Now suppose that $w$ is joined with $S$ by a dotted edge. If $w$ is also joined by a dotted edge with each of at least $(n-d+2)/2$ nodes of $\Sigma_{m'}\setminus S$,
then there are at least $(n-d+2)/2+1\ge p/2+3$ dotted edges incident to $w$, which contradicts the choice of $v_0$.
Thus, $w$ is not joined by a dotted edge with some node $x'\in\Sigma_{m'}\setminus S$. Therefore, $w$ is joined with $\l S,x'\r$ by exactly one dotted edge, so we can attach $w$ to $\l S,x'\r$ and then to $\Sigma_{m'}$ to get $\Sigma_{m'+1}$.
\end{proof}

According to the claim above, for any diagram $S'$ of a vertex $V_{S'}$ contained in $\Sigma_{m'}$, $m'\ge m$, we can attach to $\Sigma_{m'}$
any complete diagram of edge emanating from $V_{S'}$.
Walking along the edges of $P$ we can pass from $V_S$ to any other vertex of $P$. Therefore, for each elliptic diagram $S\subset\Sigma$ of order $d$
there exists some number $m_S$ such that $S\subset\Sigma_{m_S}$. Collecting diagrams of all the vertices of $P$ one by one, we obtain $\Sigma_{n-d}=\Sigma$.  
\end{proof}

\begin{lemma}
\label{alg_fix}
Given abstract Coxeter diagram $\Sigma$ with unknown weights of dotted edges, and an integer $d\ge 4$ satisfying $p\le n-d-2$, where $p$ is the number of dotted edges and $n=|\Sigma|$, there exists a finite algorithm which
\begin{itemize}
\item[1)]
verifies if there exists $P\in{\mathcal P}_d$ such that $\Sigma=\Sigma(P)$;
\item[2)]
provides the weights of dotted edges if $\Sigma=\Sigma(P)$. 
\end{itemize}
\end{lemma}

\begin{proof}
Let $\Sigma$ be an abstract Coxeter diagram with unknown weights of dotted edges. According to~\cite[Theorem~2.1 and Proposition~4.2]{V1}, $\Sigma\in\mathcal{C}_d$ if and only if its signature equals $(d,1)$ and the poset of elliptic subdiagrams of $\Sigma$ coincides with a poset of faces of some compact Euclidean $d$-polytope.   
So, the following conditions on $\Sigma$ are essential:
\begin{itemize}
\item 
the order of maximal elliptic subdiagram of $\Sigma$ equals $d$;
\item
each elliptic diagram of order $d-1$ is contained in exactly two elliptic diagrams of order $d$ (in terms of faces of $P$, each edge has exactly two vertices). 
\end{itemize}
Both conditions can be easily verified without knowing weights of dotted edges. If any of them does not hold then $\Sigma$ is not a diagram of a compact polytope.
Otherwise, we need to check, whether it is possible to assign weights to the dotted edges of $\Sigma$ to get signature $(d,1)$.

\smallskip

Now we proceed as in the algorithm constructed in Lemma~\ref{alg_p}. The only difference is in some additional computations. 
Suppose that we have already constructed a subdiagram $\Sigma_{m}\subset\Sigma$, $m\ge d+2$. This means that we know a configuration of $m+d$ vectors in $\R^{(d,1)}$ with Gram matrix corresponding to $\Sigma_{m}$. While attaching new node $w$ by expressing the corresponding vector as a linear combination of vectors corresponding to some subdiagram $\l S,x\r\subset\Sigma$ of order $d+1$, we need to compute scalar products of this vector with all the preceding $m+d$ vectors, including ones corresponding to nodes joined with $w$ by non-dotted edges. If all the scalar products coincide with ones prescribed by $\Sigma$, then we add the node to obtain $\Sigma_{m+1}$ and a configuration of $m+d+1$ vectors in $\R^{(d,1)}$, otherwise we stop the process and deduce that $\Sigma$ is not a diagram of a polytope.   
\end{proof}

\begin{lemma}
\label{alg}
There exists a finite algorithm listing Coxeter diagrams of all the polytopes from $\mathcal{P}_d$. 
\end{lemma}

\begin{proof}

\noindent
{\bf 1. General algorithm.}\ 
In fact, Lemma~\ref{alg_fix} provides a way to classify all the polytopes $P\in\mathcal{P}_d$. Indeed, according to Theorem~\ref{multi}, $\Sigma(P)$ does not contain edges of multiplicity greater than $K_0(d)-2$, and by Lemma~\ref{n}, $|\Sigma(P)|\le Q(d,K_0(d))$. The number of Coxeter diagrams (without labels of dotted edges) is finite, and for each of them we can find out whether there exists a corresponding polytope $P\in\mathcal{P}_d$. 

However, the number of diagrams we need to check in this way is extremely large even for small $K_0(d)$ and $Q(d,K_0(d))$. Below we provide a shorter algorithm which becomes realizable for reasonable $K_0(d)$. An example of application to $d=4$ (with some further refinements) is presented in Section~\ref{dim4}.  

\medskip

\noindent
{\bf 2. Reduced algorithm.} 

We proceed as in the algorithm constructed in Lemma~\ref{alg_p} by listing all possible diagrams $\Sigma_i$ on each step.

First, we list all possible diagrams $\Sigma_0=S_0$. According to Theorem~\ref{multi}, the multiplicity of edges of $\Sigma_0$ does not exceed $K_0(d)-2$, so we can list all elliptic diagrams of order $d$. Denote the list of all obtained diagrams by $L_0$. Then we construct the list $L_1$ of diagrams $\Sigma_1=\l\Sigma_0,x\r$ without dotted edges. 

After constructing any diagram $\Sigma_i$ we immediately check whether $\Sigma_i$ is already a diagram of a polytope. By Proposition~\ref{same}, a Coxeter diagram of $d$-polytope cannot contain a diagram of another $d$-polytope as a proper subdiagram. Therefore, if we get a diagram $\Sigma_i$ of $d$-polytope, we put it in the resulting list of polytopes and exclude from further considerations. 

By Lemma~\ref{dplus2}, there exists $\Sigma_2=\l\Sigma_1,y\r=\l \Sigma_0,x,y\r$ without dotted edges. In particular, the diagram $\l\Sigma_0,y\r$ should appear in $L_1$. Therefore, we do the following: for each pair $(\Sigma_1,\Sigma_1')$ of (possibly isomorphic) diagrams from $L_1$ constructed by the same $\Sigma_0$ (i.e. $\Sigma_1=\l\Sigma_0,x\r$, $\Sigma_1'=\l\Sigma_0,y\r$) we compose a diagram $\Sigma_2=\l\Sigma_0,x,y\r$ and compute the weight of the edge $\l x,y\r$ from the equation $\det(\Sigma_2)=0$. We put $\Sigma_2$ in the list $L_2$ if and only if the weight equals $\cos(\pi/k)$ for some positive integer $k\le K_0(d)$. Any $\Sigma(P)$ contains some diagram from $L_2$ as a subdiagram.

\noindent
{\bf 2.1 Attaching $\mathbf{\mathcal{M}_{\le 1}}$} 

Define the set $\mathcal{M}_{\le 1}$ as the set of nodes of future $\Sigma(P)$ joined with $\Sigma_2$ by at most one dotted edge. As we have already proved (see Step~1 of the proof of Lemma~\ref{alg_p}), after attaching of all the nodes of $\mathcal{M}_{\le 1}$ to $\Sigma_2$ we get at least $d+(n-d+2)/2$ nodes, so we obtain an inequality $$|\mathcal{M}_{\le 1}|+d+2\ge (n+d+2)/2,$$ which is equivalent to $$n\le 2|\mathcal{M}_{\le 1}|+d+2$$  
Further procedure depends on $|\mathcal{M}_{\le 1}|$ in $\Sigma(P)$ we are looking for.

If $\mathcal{M}_{\le 1}=\emptyset$, then $n=d+2$ and we are done. 

Now suppose that $|\mathcal{M}_{\le 1}|\ge 1$. The procedure to attach the nodes is the following. For each node $v\in\Sigma_1$ we join a new node $w$ with all nodes of $\Sigma_1\setminus v$ by non-dotted edges in all possible ways, and for each of them compute the weights of the edges $\l w,v\r$ and $\l w,y\r$. If one of these weights is equal to $\cos(\pi/k)$ for some positive integer $k\le K_0(d)$, and another is either equal to $\cos(\pi/k')$ or is greater than one, then we get $\Sigma_3=\l\Sigma_2,w\r$ and put it in the list $L_3$.

If we assume that $|\mathcal{M}_{\le 1}|\ge 2$, then we need to attach other nodes from $\mathcal{M}_{\le 1}$. For this we take all the pairs $(\Sigma_3,\Sigma_3')$ of (possibly isomorphic) diagrams from $L_3$ with the same $\Sigma_2$ (i.e. $\Sigma_3=\l\Sigma_2,w\r$, $\Sigma_3'=\l\Sigma_2,w'\r$), compose a diagram $\Sigma_4=\l\Sigma_2,w,w'\r$ and compute the weight of the edge $\l w,w'\r$ from the equation $\det(\l\Sigma_0,w,w'\r)=0$ (since $|\l\Sigma_0,w,w'\r|=d+2$). We put $\Sigma_4$ in the list $L_4$ if and only if the weight either is equal to $\cos(\pi/k)$ for some positive integer $k\le K_0(d)$ or is greater than one. In this case we call the pair $(\Sigma_3,\Sigma_3')$ {\it compatible}. So, the list $L_4$ consists of unions of pairs of compatible diagrams from $L_3$. 

In the same way we can construct a list $L_5$ as the set of triples of mutually compatible diagrams from $L_3$, and so on. In other words, if we are looking for $\Sigma(P)$ with $|\mathcal{M}_{\le 1}|=m'-2$, then $\Sigma(P)$ should contain a subdiagram from the list $L_{m'}$ composed of $(m'-2)$-tuple of mutually compatible diagrams from $L_3$. 
Notice that $m'\le Q(d,K_0(d))-d$, so the number of lists $L_i$ is finite.

\noindent
{\bf 2.2 Walking along edges}\ 

Denote by $m$ the maximal index of $L_i$ obtained at the previous step, i.e. $L_{m+1}=\emptyset$ but $|L_{m}|\ge 1$.  
We fix $n$ and look for polytopes $P\in \mathcal{P}_{(d,n)}$. Notice that we have a bound on $n$: 
$$n\le 2(m-2)+d+2=2m+d-2$$ 

Define the list $L'= \bigcup\limits_{i=3}^{m}{L_i}$. Let $L$ be the subset of $L'$ consisting of all diagrams which are not diagrams of $d$-polytopes and which contain no more than $n-d-2$ dotted edges.  

Consider any diagram $\Sigma_i\in L$. $\Sigma_i$ is not a diagram of a polytope, so it contains at least one elliptic subdiagram $S$ of order $d$ which belongs to less than $d$ complete diagrams of edges. We will attach to diagram $\Sigma_i$ a node $v$ to create a new   complete diagram of an edge in the following way: for each $u\in \Sigma_i\setminus S$ we attach $v$ to $\l S,u\r$ with at most one dotted edge, and then compute the weights of all the other edges joining $v$ with $\Sigma_i$. We need to attach nodes with at least $2$ dotted edges only, otherwise we get a diagram $\Sigma_{i+1}\in L_{i+1}$. Here we assume that the diagram of the initial vertex $S_0$ contains a node incident to a maximal number of dotted edges, which implies that $v$ is not joined by a dotted edge with some node $u\in \Sigma_i\setminus S$ (the proof is identical to the one provided in the claim in Step~2 of Lemma~\ref{alg_p}).  
 
Let $L^1$ be the set of all such diagrams $\l \Sigma_i,v\r$ for all $\Sigma_i\in L$ and $\l S,u\r\subset\Sigma_i$  
that contain no more than $n-d-2$ dotted edges.   

Creating in this way complete diagrams of edges and attaching them to all $\Sigma_{i+1}\in L^1$, we obtain the set $L^2$. In the same way we get the sets $L^3$, $L^4$ and so on (we attach a new node to some $\Sigma_i$ only if  $\Sigma_i$ is not a diagram of a $d$-polytope). The union $(\bigcup\limits_{i\le\frac{n-d-2}{2}}{L^i})\cup L_*$  contains  diagrams $\Sigma(P)$ of all the polytopes $P\in{\mathcal P}_{(d,n)}$ (where $L_*$ is the list of all $d$-polytopes obtained at the Step~2.1).
\end{proof}

\begin{remark}
\label{fast}
Even the reduced algorithm is not too fast. Below we list some method to obtain better estimates for $K_0(d)$
and to improve the algorithm.

\smallskip

\noindent
\textbf{1.} Better estimate for $K_0(d)$.\\
 Let $\Sigma$ be a diagram of a polytope $P\in {\mathcal P}(d,n)$ and let 
$k\le K_0(d)-2$ be the maximal multiplicity of an edge in $\Sigma$.
Let $S\subset \Sigma$ be an elliptic subdiagram of order $d$ containing an edge of multiplicity $k$.
By Lemma~\ref{dplus2},  there exist two points $x,y$ such that $\l S,x,y \r\in \Sigma$ contains no dotted edges
and all edges in $\l S,x,y \r$ have multiplicity less or equal to $k$.

This gives us a way to obtain a better estimate for the maximal multiplicity of an edge.
We consider all elliptic diagrams $S$ of order $d$ containing no edges of multiplicity greater than $K_0(d)-2$
and let $k(S)$ be the maximal multiplicity of an edge in $S$.
For each of these diagrams $S$ we consider all indefinite diagrams  $\l S,x,y \r$ with zero determinant 
containing neither dotted edges nor edges of multiplicity greater than $k(S)$.
The number of these diagrams may not be too large. Then the maximal multiplicity amongst ones  appearing in all 
these diagrams can be taken into account instead of $K_0(d)-2$.

\smallskip

\noindent
\textbf{2.} Changing starting vertex.\\
In the algorithm provided in the proof of Lemma~\ref{alg} we start from diagram $S_0$ of a vertex satisfying some requirements. 
More precisely, $S_0$ should contain a node $v_0$ incident to maximal number of dotted edges. However, sometimes this condition can be avoided. 

Suppose that we have already attached all the nodes from  $\mathcal{M}_{\le 1}$ to get $\Sigma_{i}$. Now we want to attach to $\Sigma_{i}$ a new node $v$ to get a complete diagram of some edge $\l S,v\r$ for diagram of some vertex $S$. 
As we have seen above, 
$$n\le 2i+d-2,$$
so we have an inequality 
$$p\le n-d-2\le 2i-4$$
If we assume that $\Sigma_{i}$ already contains at least $i-4$ dotted edges, then any new node $v$ is incident to at most $i$ dotted edges, so there exists $u\in\Sigma_i\setminus S$ which is not joined with $v$ by a dotted edge. Therefore, we can apply our algorithm.

The requirement for $\Sigma_{i}$ to contain at least $i-4$ dotted edges is not too restrictive. Since $|\mathcal{M}_{\le 1}|=i-2$, this condition holds, for example, if there are no triples $(\Sigma_3,\Sigma_3',\Sigma_3'')$ of mutually compatible diagrams such that  $\l\Sigma_3,\Sigma_3',\Sigma_3''\r$ contains no dotted edges. But while avoiding the requirement on $S_0$, we hugely reduce the computations. Indeed, we can now start from diagram  $S_0$ of any vertex. In particular, we may assume that  $S_0$ contains an edge of the maximal multiplicity in $\Sigma$. This allows us in each case to make calculations up to some multiplicity depending on the starting elliptic diagram, but not up to $K_0(d)-2$.      
\end{remark}

\section{Polytopes in dimension $4$}
\label{dim4}

In this section, we apply the algorithm provided in Lemma~\ref{alg} (taking into account Remark~\ref{fast}) to classify polytopes from $\mathcal{P}_4$.
We list the main steps and further essential refinements.

Suppose that $\Sigma=\Sigma(P)$ is a Coxeter diagram of a polytope, and consider an edge $S_0$ of maximal multiplicity $k-2$ in $\Sigma$. According to the algorithm, we want to find a subdiagram $\l S,x\r\subset\o S_0$ and $y\in\Sigma$ satisfying the assumptions of Lemma~\ref{local} and to estimate its local determinant. $P(S_0)$ is a polygon, so it is either a triangle, or has a pair of non-intersecting sides. We consider these cases separately.

\noindent
{\bf Case 1.} \ $P(S_0)$ has more than $3$ sides, $k\ge 7$.\\ In this case its diagram $\Sigma_{S_0}$ contains a subdiagram described in Case~1 of Lemma~\ref{separate} with a dotted edge of weight at least $\sqrt{3/2}$. We introduce the notation on Fig.~\ref{poly}: $\rho$ is the weight of the dotted edge, and the remaining variables are integers.

\begin{figure}[!h]
\begin{center}
\psfrag{ro}{\footnotesize $\rho$}
\psfrag{k}{\footnotesize $k$}
\psfrag{l}{\footnotesize $l$}
\psfrag{m}{\footnotesize $m$}
\psfrag{p}{\footnotesize $p$}
\psfrag{q}{\footnotesize $q$}
\psfrag{s}{\footnotesize $s$}
\psfrag{t}{\footnotesize $t$}
\psfrag{u}{\footnotesize $u$}
\psfrag{x}{\footnotesize $x$}
\psfrag{y}{\footnotesize $y$}
\psfrag{S0}{\footnotesize $S_0$}
\psfrag{S}{\footnotesize $S$}
\epsfig{file=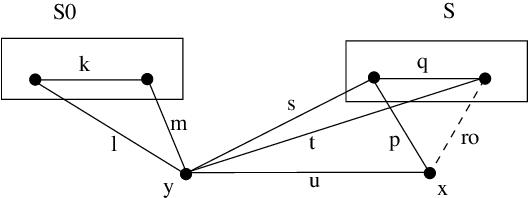,width=0.545\linewidth}
\caption{A subdiagram $\l S_0,S,x,y\r\subset\Sigma$, the case $P(S_0)$ is not a triangle}
\label{poly}
\end{center}
\end{figure}

Writing down the explicit formula of $\det(\l S,x,y\r,y)$, we can see that the local determinant is a decreasing function of $\rho$ (in the appropriate domain of all the variables). Therefore, while looking for the estimate we may assume that $\rho=\sqrt{3/2}$. We get the following expression:
$$\det(\l S,x,y\r,y)\le 1+\frac{R(p,q,s,t,u)}{-\frac{1}{2}-\cos^2{\frac{\pi}{q}}-\sqrt{6}\cos{\frac{\pi}{p}}\cos{\frac{\pi}{q}}-\cos^2{\frac{\pi}{p}}},\eqno{(*)}$$
where
\begin{multline*}
R(p,q,s,t,u)=\frac{1}{2}\cos^2\frac{\pi}{s}-2\cos{\frac{\pi}{s}}\cos{\frac{\pi}{q}}\cos{\frac{\pi}{t}}-\sqrt{6}\cos{\frac{\pi}{s}}\cos{\frac{\pi}{q}}\cos{\frac{\pi}{u}}-\\
-\sqrt{6}\cos{\frac{\pi}{s}}\cos{\frac{\pi}{p}}\cos{\frac{\pi}{t}}-2\cos{\frac{\pi}{s}}\cos{\frac{\pi}{p}}\cos{\frac{\pi}{u}}-2\cos{\frac{\pi}{t}}\cos{\frac{\pi}{p}}\cos{\frac{\pi}{u}}\cos{\frac{\pi}{q}}-\\
-\sqrt{6}\cos{\frac{\pi}{t}}\cos{\frac{\pi}{u}}-\cos^2{\frac{\pi}{t}}\left(1-\cos^2{\frac{\pi}{p}}\right)-\cos^2{\frac{\pi}{u}}\left(1-\cos^2{\frac{\pi}{q}}\right)
\end{multline*}
Now we will use the fact that every unknown weight is less than one. A rough estimate (all the entries in the denominator and positive ones in the numerator are zeros, and negative ones in the numerator are equal to $1$) gives an inequality  $\det(\l S,x,y\r,y)\le 17+6\sqrt{6}$. Applying Proposition~\ref{loc_sum} and taking into account that $\det(\l S_0,S,x,y\r)=0$, we see that $\det(\l S_0,y\r,y)\ge -20-8\sqrt{6}$. Since the local determinant $\det(\l S_0,y\r,y)$ is a decreasing function on all the weights, this implies that $\d(2,3,k)\ge -16-6\sqrt{6}$, which, in its turn, implies that $k\le 35$. 

The bound above is very rough. We can improve it in the following way. Suppose that $p=2$. Then, repeating the considerations above for the simplified expression, we obtain a bound $k\le 27$. By symmetry, we get the same if $q=2$. Applying the same procedure to the case $t=2$ (or $u=2$), we obtain an upper bound $k\le 26$. Thus, if we assume now that $k\ge 28$, then each of $p,q,t,u$ should exceed $2$. Differentiating the expression $({*})$ with respect to $p$ we see that the local determinant decreases. This implies that maximum is attended for $p=3$, and, similarly, for $q=3$. Repeating the procedure for simplified expression we get an estimate $k\le 9$, so the contradiction shows that $k\le 27$, and if $k>9$ then at least one of $p,q,t,u$ is equal to $2$.

Let us now improve the estimate another one time. Assume that $q=2$. Then we have two possibilities: either $p=2$ or $p>2$. Applying the same procedure in the both cases we get an estimate $k\le 19$. In other words, if we assume that $k>19$, then $p,q>2$. Suppose that $k>19$. Suppose also that $t>2$. Then the derivative of the local determinant with respect to $p$ occurs to be negative, which implies $k\le 9$ as we have seen above. Thus, if $k> 19$, then $p,q>2$, and $t=2$. Similarly, $u=2$. Computing the local determinant in these settings, we get a bound $k\le 7$. The contradiction shows that  if $P(S_0)$ is not a triangle then $k\le 19$.

\medskip

Now we make the following computations. For each $8$-tuple of integers $(k,l,m,p,q,s,t,u)$ (where $7\le k\le 19$ is the maximal one) we compute $\rho$ to satisfy the equation $\det(\l S_0,S,x,y\r)=0$, and list all of them satisfying $\rho\ge\sqrt{3/2}$. If $P\in\mathcal{P}_4$ and $P(S_0)$ is not a triangle, then $\Sigma(P)$ contains one of these subdiagrams. The list (denote it by $\mathcal{L}$) is large: it contains around $167000$ diagrams. The number is approximate since while running the program we allow $\rho$ to be a bit smaller to prevent errors.
Now it is easy to see that there exists a node $y_1\in\Sigma\setminus\l S_0,S,x,y\r$ which is not joined with $\l S_0,S,x,y\r$ by dotted edges (the proof is similar to Lemma~\ref{dplus2}). Thus, the diagram $\l S_0,S,x,y_1\r$ should be contained in the list $\mathcal{L}$. More precisely, we are looking for a pair of diagrams from $\mathcal{L}$ parametrized by $(k,l,m,p,q,s,t,u,\rho)$ and $(k,l_1,m_1,p,q,s_1,t_1,u_1,\rho)$, such that the diagram  $\l S_0,S,x,y,y_1\r$ has signature $(4,1,2)$, and the nodes  $y$ and $y_1$ are not joined by a dotted edge, see Fig.~\ref{dot4}.
    
\begin{figure}[!h]
\begin{center}
\psfrag{w}{\footnotesize $w$}
\psfrag{l'}{\footnotesize $l'$}
\psfrag{m'}{\footnotesize $m'$}
\psfrag{y'}{\footnotesize $y'$}
\psfrag{s'}{\footnotesize $s'$}
\psfrag{t'}{\footnotesize $t'$}
\psfrag{u'}{\footnotesize $u'$}
\psfrag{ro}{\footnotesize $\rho$}
\psfrag{k}{\footnotesize $k$}
\psfrag{l}{\footnotesize $l$}
\psfrag{m}{\footnotesize $m$}
\psfrag{p}{\footnotesize $p$}
\psfrag{q}{\footnotesize $q$}
\psfrag{s}{\footnotesize $s$}
\psfrag{t}{\footnotesize $t$}
\psfrag{u}{\footnotesize $u$}
\psfrag{x}{\footnotesize $x$}
\psfrag{y}{\footnotesize $y$}
\psfrag{S0}{\footnotesize $S_0$}
\psfrag{S}{\footnotesize $S$}
\epsfig{file=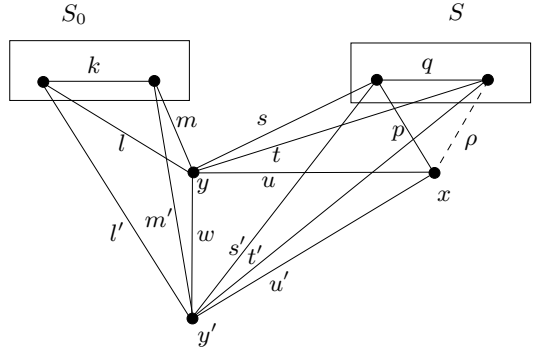,width=0.545\linewidth}
\caption{A subdiagram $\l S_0,S,x,y,y_1\r\subset\Sigma$, the case $P(S_0)$ is not a triangle}
\label{dot4}
\end{center}
\end{figure}

So, we look for pairs of diagrams from  $\mathcal{L}$ with the same $k,p,q$, and $\rho$, and for every such pair compute the weight of the edge $yy_1$. 
The result is empty list.

Therefore, for each polytope from $\mathcal{P}_4$ with $S_0$ being the edge of maximal multiplicity $k-2\ge 5$ the polygon $P(S_0)$ is a triangle.

\noindent                 
{\bf Case 2.} \ $P(S_0)$ is a triangle, $k\ge 7$. \\ 
In this case $\l S,x\r=\Sigma_{S_0}$, and $y$ is any vertex of $\Sigma$ which is not joined with $\l S_0, \Sigma_{S_0}\r$ by a dotted edge (it does exist by Lemma~\ref{dplus11}). We get the subdiagram shown on Fig~\ref{triangle}.

\begin{figure}[!h]
\begin{center}
\psfrag{r}{\footnotesize $r$}
\psfrag{r2}{\footnotesize $\rho_2$}
\psfrag{k}{\footnotesize $k$}
\psfrag{l}{\footnotesize $l$}
\psfrag{m}{\footnotesize $m$}
\psfrag{p}{\footnotesize $p$}
\psfrag{q}{\footnotesize $q$}
\psfrag{s}{\footnotesize $s$}
\psfrag{t}{\footnotesize $t$}
\psfrag{u}{\footnotesize $u$}
\psfrag{x}{\footnotesize $x$}
\psfrag{y}{\footnotesize $y$}
\psfrag{S0}{\footnotesize $S_0$}
\psfrag{S}{\footnotesize $S$}
\epsfig{file=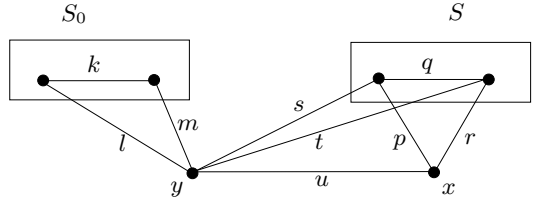,width=0.545\linewidth}
\caption{A subdiagram $\l S_0,S,x,y\r\subset\Sigma$, the case $P(S_0)$ is a triangle}
\label{triangle}
\end{center}
\end{figure}

First, we are looking for reasonable $K_0(4)$.
The considerations are very similar to the previous case, but are much more detailed, so we show the plan. We may assume that $p\le q\le r$. 
The first attempt to get a bound is to consider two cases: either $p=2$ or $p>2$. If $p>2$ then a very rough estimate gives $\det(\l S,x,y\r,y)\le 1-15/\det(L_{3,3,4})$, where $15$ is the number of negative summands in the denominator of the expansion of the expression for local determinant  $\det(\l S,x,y\r,y)$, and $L_{3,3,4}$ is a diagram of triangle with minimum possible determinant (while $p>2$). This implies that $k\le 41$. Doing the same for $p=2$ we get a bound $\det(\l S,x,y\r,y)\le 1-9/\det(L_{2,3,7})$, which implies  $k\le 76$.

At this point we could stop and say that $K_0(4)=76$. However, this estimate does not look reasonable. Making several similar refinements (considering more detailed cases for triples $(p,q,r)$) we can improve the bound. For example, subdividing the case $p=2$ into $q=3$ and $q>3$ we get an estimate $k\le 65$ for the first case and $k\le 48$ for the latter. Therefore, we see that $K_0(4)\le 65$, and if $k>48$ then $p=2$ and $q=3$. According to our algorithm (see Remark~\ref{fast}) we run the following computation for $48<k\le 65$ in these settings: we are looking for the diagram $\l S_0,y,\Sigma_{S_0}\r$ whose determinant vanishes, and all the entries do not exceed $k$. Actually, we find nothing, which implies that $K_0(4)\le 48$. Repeating the procedure for more detailed cases of triples $(p,q,r)$ and running short computations, we come to a bound $K_0(4)\le 30$.

In fact, the estimate $K_0(4)=30$ is sharp: there exists a unique diagram $\l S_0,y,\Sigma_{S_0}\r$ with zero determinant and $k=30$, it is shown on Fig.~\ref{k30}. Here $(p,q,r)=(2,3,15)$.

\begin{figure}[!h]
\begin{center}
\psfrag{x}{\footnotesize $x$}
\psfrag{y}{\footnotesize $y$}
\psfrag{30}{\footnotesize $30$}
\psfrag{15}{\footnotesize $15$}
\epsfig{file=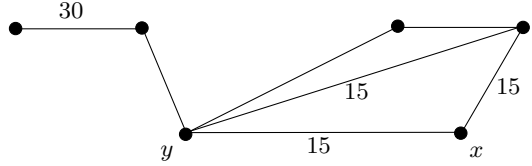,width=0.545\linewidth}
\caption{A subdiagram $\l S_0,y,\Sigma_{S_0}\r\subset\Sigma$, with $k=30$}
\label{k30}
\end{center}
\end{figure}
 
\bigskip

So, we found $K_0(d)$. Now we go to the first step of the reduced algorithm (see Lemma~\ref{alg}): we want to list all the diagrams $\Sigma_2=\l S_0,y,\Sigma_{S_0}\r$.
We find one diagram with $k=30$ (mentioned above), one with $k=20$, and $79$ with $7\le k\le 18$. Then we proceed exactly as it is prescribed by the reduced algorithm.
On our way we get all the Esselmann polytopes and $4$-polytopes with $7$ facets with multi-multiple edges. However, we do not get anything new.

\noindent                 
{\bf Case 3.} \ $k\le 6$. \\ 
In this case we follow the reduced algorithm {\it verbatim}. Again, we meet all the remaining Esselmann polytopes and $4$-polytopes with $7$ facets, but we do not find any new polytope. This leads to the following theorem which, due to results of~\cite{L,Ess1,T}, completely describes $\mathcal{P}_4$:

\begin{theorem}
\label{d4}
All the polytopes belonging to $\mathcal{P}_4$ have at most $7$ facets. 
\end{theorem}

\begin{remark}
The way to deal with the diagrams such that $P(S_0)$ contains dotted edge can be easily generalized to higher dimensions. If the estimate for maximal multiplicity in this case is much better than in general one (as we have for $d=4$), this may speed up the process.
\end{remark}

\section{Essential polytopes}
\label{essential}

In section~\ref{ess_def} we define the notion of essential polytope, present the list of known (to the authors) compact Coxeter polytopes and formulate the main result of this section,
Theorem~\ref{ess-t}.
In  section~\ref{ess_prop} we discuss the methods allowing to differ essential polytopes from non-essential ones.
In   section~\ref{ess_proof} we use these methods to show that all the polytopes claimed to be essential in Theorem~\ref{ess-t} are essential,
and that all other known polytopes (with possible exclusion of two polytopes listed in Question~\ref{qop}) are not essential. 

\subsection{Definition of essential polytope}
\label{ess_def}

Let $P$ be a hyperbolic Coxeter polytope. We define two elementary gluing operations in the following way.

Denote by $G_P$ the reflection group generated by reflections in the facets of $P$. Suppose that there exists a finite index subgroup $H\subset G_P$ generated by reflections.
Then the fundamental polytope $P'$ of $H=G_{P'}$ consists of several copies of $P$. As it is proved in~\cite{All}, this procedure gives rise to infinitely many Coxeter polytopes of finite volume in all dimensions up to $19$, and infinitely many compact Coxeter polytopes in all dimensions up to $6$. Conversely, taking a polytope $F$ such that $G_P\subset G_F$, we obtain a tessellation of $P$ by copies of $F$ (in this case we say that {\it $P$ is decomposed with fundamental polytope $F$}).
Clearly, a polytope minimal with respect to this procedure is a fundamental domain of maximal reflection group.

Another operation is the following. Suppose that for some facet $f$ of $P$ there exists a Coxeter polytope $P'$ with a facet $f'$ congruent to $f$, such that sums of corresponding dihedral angles composing by $f$ and $f'$ respectively are integral divisors of $\pi$. Then, gluing $P$ and $P'$ along $f$ and $f'$, we get a new Coxeter polytope. Makarov~\cite{M} used this procedure to construct infinite series of polytopes in hyperbolic spaces of dimension $4$ and $5$. In particular, he produced fundamental domains of infinite number of non-commensurable reflection groups. It is easy to describe the converse operation, i.e. {\it dissecting} a Coxeter polytope into two smaller ones. A compact Coxeter polytope $P$ can be divided into two smaller Coxeter polytopes $P_1$ and $P_2$ if there exists a hyperplane $\alpha$ dissecting $P$ such that for every facet $f$ of $P$ 

1) if $\alpha$ intersects $f$ at interior point of $f$ then $\alpha$ is orthogonal to $f$;

2) if $\alpha\cap f$ contains no interior points of $f$ and $\alpha\cap f\ne \emptyset$ then $\alpha\cap f=f\cap f'$ for some facet $f'$ of $P$, and the dihedral angles composed by $\alpha$ with $f$ and $f'$ are integral divisors of $\pi$.               

In this case we say that $P$ is {\it dissected } into Coxeter polytopes $P_1$ and $P_2$.

\begin{defin}
We say that a Coxeter polytope $P$ is {\it essential} if it is minimal with respect to both operations above, i.e. the corresponding reflection group  $G_P$ is maximal, 
and $P$ cannot be dissected by a hyperplane into two smaller Coxeter polytopes.
\end{defin}

Clearly, the classification of compact Coxeter hyperbolic polytopes can be reduced to classification of essential ones. We show that all known compact Coxeter polytopes of dimension greater than $3$ can be glued from a finite number of essential ones, and list all known essential compact Coxeter polytopes. The main question we want to ask is

\begin{question}
Is the number of essential compact hyperbolic Coxeter polytopes of dimension greater than $3$ finite? 
\end{question} 

There are two main sources of Coxeter polytopes. One is the result of combinatorial considerations. In this way the following lists of compact polytopes were obtained: 
\begin{itemize}
\item simplices~\cite{L} (in other words, $d$-polytopes with $d+1$ facets); 
\item $d$-polytopes with $d+2$ facets: 
\begin{itemize}
\item[--] simplicial prisms~\cite{K}; 
\item[--] Esselmann $4$-polytopes, i.e. products of two triangles~\cite{Ess2};
\end{itemize}
\item products of simplices~\cite{Ess2} (which appears to be a sublist of the previous list); 
\item $d$-polytopes with $d+3$ facets~\cite{Ess1},~\cite{T}; 
\item Napier cycles~\cite{ImH} (which is a sublist of the previous list); 
\item truncated simplices~\cite{S}. 
\end{itemize}

Another source is arithmetic construction. 
Fundamental domains of cocompact arithmetic reflection groups were computed in~\cite{Bu1},~\cite{Bu2}, and~\cite{Bu3}.    
Amongst them there are:

\begin{itemize}
\item $8$-dimensional polytope with $8+3=11$ facets (the only polytope in dimension $d\ge 8$ known so far);
\item $7$-dimensional polytope $7+4=11$ facets (besides its double, the only polytope known in dimension $7$);
\item $6$-dimensional polytope with $34$ facets (see~\cite[Table~2.2]{All} where the Gram matrix of this polytope is reproduced).
\end{itemize} 

To the authors best knowledge, all the other examples of compact Coxeter hyperbolic polytopes are defined by gluing smaller ones (see e.g.~\cite{M},~\cite{All},~\cite{PV}), 
but we do not care about them now since they are already glued. 

Now we want to find out which polytopes from the lists above are essential. 
We will refer to the list above as to the ``list of known polytopes''.

\begin{theorem}
\label{ess-t}
All the polytopes listed in Table~\ref{ess-list} are essential. All the remaining known compact hyperbolic Coxeter polytopes of dimension at least $4$ except ones listed in Question~\ref{qop} can be glued from these ones and the polytopes listed in Question~\ref{qop}.    
\end{theorem}

The following question remains open:

\begin{question}
\label{qop}
Is the polytope shown on Fig.~\ref{open}, as well as the $6$-polytope with $34$ facets, essential? 
\end{question}

\begin{figure}[!h]
\begin{center}
\psfrag{r2}{\footnotesize $\rho_2$}
\epsfig{file=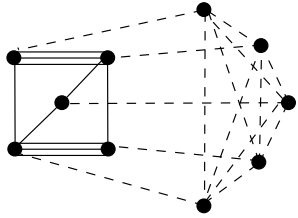,width=0.3045\linewidth}
\caption{$5$ times truncated $4$-simplex, one of the two polytopes not proven to be essential}
\label{open}
\end{center}
\end{figure}


\begin{conj}
\label{conjecture}
The number of essential compact hyperbolic Coxeter polytopes of dimension at least four is finite.
\end{conj}

\bigskip

\subsection{Properties of non-essential polytopes}
\label{ess_prop}

The main tool in our consideration is the following result:

\begin{theorem}[\cite{subgr}, Theorem~1.2]      
\label{sub}
Let $G$ be an infinite indecomposable Coxeter group, and let $H\subset G$ be a finite index reflection subgroup. Then the rank of $G$ does not exceed the rank of $H$.
\end{theorem}
Here by {\it rank} of reflection group we mean the number of reflections in standard generating set, i.e. in Coxeter system.

Applying this to the case of cocompact hyperbolic reflection groups, we get

\begin{cor}[\cite{subpol}, Theorem~1]
\label{subp}
If $G_P\subset G_F$, then the number of facets of $F$ does not exceed the number of facets of $P$.
\end{cor}

A {\it nerve} of a Coxeter group is a simplicial complex with vertices indexed by standard generators; a set of vertices span a simplex if the corresponding reflections 
generate a finite group. In case of equal numbers of facets we have the following result.

\begin{theorem}[\cite{subgr}, Lemma~5.1]
\label{eq}
Let $G$ be an infinite indecomposable Coxeter group, and let $H\subset G$ be a finite index reflection subgroup of the same rank as $G$. Then the nerve of $H$ can be obtained from the nerve of $G$ by deleting some simplices.  
\end{theorem}

Reformulating Theorem~\ref{eq} in terms of compact Coxeter polytopes, we obtain

\begin{cor}
\label{eqp}
Let $P$ and $F$ be compact hyperbolic Coxeter polytopes of the same dimension. Suppose that $G_P\subset G_F$ and the numbers of facets of $F$ and $P$ are equal. Then $F$ and $P$ are combinatorially equivalent. 
\end{cor}

We can reformulate the corollary in terms of simplicial polytopes. Recall that a fundamental polytope of cocompact reflection group is simple. Therefore, 
for a cocompact reflection group $G_F$ its nerve coincides with face-lattice of simplicial polytope $F'$ dual to $F$. Hence, in view of Theorem~\ref{eq}, 
Corollary~\ref{eqp} is equivalent to the following statement. 

\begin{lemma}
\label{esm}
Let $F'$ be a simplicial $d$-polytope, and let $\mathcal F$ be its face-lattice. Then any proper simplicial subcomplex of $\mathcal F$ with the same number of vertices is not a face-lattice of a simplicial $d$-polytope.
\end{lemma}

\begin{proof}
Indeed, suppose that $\mathcal F_1$ is a proper subcomplex of $\mathcal F$, and $\mathcal F_1$ is a face-lattice of a simplicial polytope $F_1'$. Then both $\mathcal F_1$ and $\mathcal F$ can be thought as triangulations of a $(d-1)$-sphere. But since $\mathcal F$ contains $\mathcal F_1$ and they do not coincide, the number of vertices of $\mathcal F$ should exceed the number of vertices of $\mathcal F_1$, so we come a contradiction.
\end{proof}

Now we are able to list some properties of decompositions of polytopes into smaller ones.
Lemma~\ref{pr_s} concerns the case of a finite index subgroup (i.e., decomposition), and Lemma~\ref{pr_d} concerns the case of dissection.

For any polytope $P$ we denote by $\Vol(P)$ and  by $|P|$  its volume and the number of its facets respectively.
By a {\it standard subgroup} of $G_P$ we mean a subgroup generated by reflection with respect to some of the facets of $P$.

\begin{lemma}
\label{pr_s}
Let $P$ and $F$ be compact hyperbolic Coxeter polytopes of the same dimension. Suppose that $G_P\subset G_F$. Then

$(1)$ $|F|\le |P|$;

$(2)$ if  $|F|=|P|$, then $P$ is combinatorially equivalent to $F$;

$(3)$ every standard finite subgroup of $G_P$ is a subgroup of some standard finite subgroup of $G_F$;

$(4)$ $\Vol(P)=\Vol(F)\left[G_F:G_P\right]$;
\end{lemma}

\begin{proof}
Statements $(1)$ and $(2)$ follow from Corollaries~\ref{subp} and~\ref{eqp} respectively. Statement~$(4)$ is evident, and $(3)$ is proved in~\cite{simplex}. 
\end{proof}

\begin{lemma}
\label{pr_d}
Let $P$ be a compact hyperbolic Coxeter polytope, and suppose that $P$ can be dissected by a hyperplane $\alpha$ into two Coxeter polytopes $P_1$ and $P_2$. Then 

$\mathrm{(I)}$  $|P_i|\le |P|$, for $i=1,2$;

$\mathrm{(II)}$ if $|P_1|=|P|$, then $P_1$ is combinatorially equivalent to $P$.

$\mathrm{(III)}$ every dihedral angle of $P$ either coincides with a dihedral angle of some of $P_i$, or is tessellated by one of angles $P_1$ and one of angles of $P_2$;

$\mathrm{(IV)}$ $\Vol(P)=\Vol(P_1)+\Vol(P_2)$.
\end{lemma}

\begin{proof}
To prove $\mathrm{(I)}$, consider the group $G_{P_i}$. If $P_i$ has more facets than $P$ does, then $\alpha$ intersects interior of all the facets of $P$, 
so $G_{P_i}$ contains $G_{P}$. Since $P_i$ is a Coxeter polytope, this contradicts  Corollary~\ref{subp}, so $\mathrm{(I)}$ is proved.     
Statements~$\mathrm{(III)}$ and~$\mathrm{(IV)}$ are evident.

To prove $\mathrm{(II)}$,
notice first that $P_2$ has a unique facet $\beta$ which is not a facet of $P_1$. 
Let $f_1, \dots, f_{n-1},\beta$ be all the facets of $P$ and $\bar f_i$ be a hyperplane containing $f_i$. 
It is sufficient to prove that for any subset $I\in \{1,\dots, n-1\}$ an intersection $\bigcap\limits_{i\in I} f_i$  intersects $\alpha$ 
if and only if it intersects $\beta$.
Without loss of generality we may consider  maximal intersections only, namely the intersections of $d-1$ facets, or, in other words, the edges.
Notice also, that any edge of $P$ which does not belong to $\beta$ is an edge of $P_1$ (or an edge lying on a line that contains an edge of $P_1$).

Let $AB$ be any edge of $P_2$ such that $A\in \alpha$ and $B\notin \alpha$. Then $B\in \beta$ (otherwise $B$ is an intersection of $d$ faces of $P_1$, 
so $B\in P_1$,  by  Andreev's  theorem~\cite{An0}).
Hence, $AB\subset P_1$, which implies $AB\subset \alpha$ in contradiction to the assumption.   
This implies that if some edge of $P$ intersects $\alpha$ then it intersects $\beta$.
On the other hand, if some edge $XY$ of $P$ intersects $\beta $ ($X\in \beta$)  and $XY\not\subset \beta$,
then $XY\subset l$ where $l$ is a line containing some edge of $P_1$, so $l$ intersects $\alpha$ and the condition   $\mathrm{(II)}$ is proved. 
\end{proof}

\medskip

\noindent{\bf An algorithm for dissections.}
For any given Coxeter $d$-polytope $P$, Lemma~\ref{pr_d} allows to check whether $P$ can be dissected into two Coxeter polytopes. 

To check if $P$ can be dissected in ``orthogonal way'' (i.e., by a hyperplane orthogonal to all facets of $P$ which it does meet)
we do the following. For every $d$-tuple of facets $(f_1,\dots,f_d)$ of $P$ (except those composing vertices of $P$) we draw a hyperplane orthogonal to all the facets $f_1,\dots,f_d$ (it is unique), and look at the angles of $\alpha$ with other facets. If all of them are right for some $d$-tuple of facets, we get a dissection. 

To check if $P$ can be dissected in ``non-orthogonal way'', for every $(d-2)$-face of $P$ we draw a hyperplane $\alpha$ dissecting the corresponding dihedral angle into two, each of those is integral divisor of $\pi$, and compute the angles of $\alpha$ with the remaining facets of $P$. Clearly, the number of such hyperplanes is finite since one of the new dihedral angles is not less than half of the initial one. If all the new angles are integral divisors of $\pi$, we obtain a dissection.

However, conditions $\mathrm{(I)}-\mathrm{(IV)}$ of Lemma~\ref{pr_d} usually give a much shorter way to understand if $P$ can be dissected or not. Suppose that $P$ is a union of $P_1$ and $P_2$. Due to $\mathrm{(I)}$ and existing classifications of polytopes with small number of facets, we usually know the complete list of candidates to be $P_1$ and $P_2$ (the only exceptions are the simplex truncated  $5$ times, and $6$-polytope with $34$ facets we discuss later). Condition $\mathrm{(II)}$ further reduces the number of candidates. Then we glue pairs of candidates along congruent facets with appropriate dihedral angles and look if we get $P$. Condition $\mathrm{(IV)}$ also allows to reduce the number of pairs in consideration when we know volumes of polytopes (e.g. in even-dimensional case due to Poincar\'e formula~\cite[Part I, Chapter 7]{29} or for simplices~\cite{JKRT1}).     

\medskip

Similarly, Lemma~\ref{pr_s} helps to check if $P$ is a fundamental domain of a subgroup of some group $G_F$. However, here we do not have an algorithm providing the answer. For example, we cannot decide if $6$-polytope with $34$ facets found by Bugaenko~\cite{Bu2} (the corresponding set of normal vectors is reproduced in~\cite[Table~2.1]{All}) is essential or not: we do not know the lists of $6$-polytopes with $10$ or more facets.

To simplify the considerations concerning dissections we will also use the following technical lemma.


For any abstract Coxeter diagram $\Sigma$ denote by $\Sigma_{\mathrm{odd}}$ the diagram obtained from $\Sigma$ by removing all nodes which are not incident to
any edge labeled by an odd number $k$ (or to a simple or triple edge) 
(see Fig.~\ref{odd}).

\begin{figure}[!h]
\begin{center}
\begin{tabular}{rlp{1.4cm}rl}
$\Sigma$
&
\psfrag{10}{\scriptsize $10$}
\epsfig{file=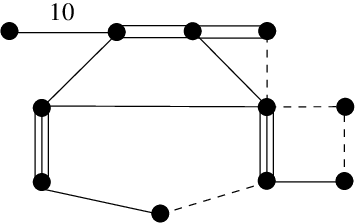,width=0.3\linewidth}
&&
$\Sigma_{\mathrm{odd}}$
&
\epsfig{file=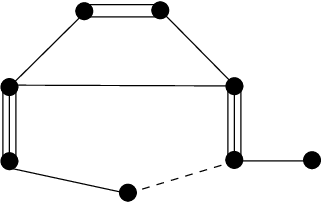,width=0.27\linewidth}
\end{tabular}
\caption{The diagram on the right is $\Sigma_{\mathrm{odd}}$ for the abstract diagram $\Sigma$ on the left}
\label{odd}
\end{center}
\end{figure}

\begin{lemma}
\label{dis_subd}
Suppose that $P$ can be dissected into two Coxeter polytopes $P_1$ and $P_2$,
both in the list of currently known polytopes. 
Then for any subdiagram $\Sigma'\subset \Sigma(P_1)$ such that $\Sigma'_{\mathrm{odd}}=\Sigma'$ 
the inclusion $\Sigma'\subset \Sigma(P)$ holds.
\end{lemma}

\begin{proof}
 The dihedral angles of all currently known polytopes are in the set $\pi/2$, $\pi/3$, $\pi/4$, $\pi/5$,
$\pi/8$ and $\pi/10$. Let $\alpha$ be the dissecting hyperplane and
suppose that  $\alpha$ dissects some dihedral angle of $P$ formed by facets $\alpha_1$ and $\alpha_2$. 
Clearly,  $\angle (\alpha_1,\alpha_2)=\angle (\alpha,\alpha_1)+\angle (\alpha,\alpha_2)$ (where ($\angle(\alpha,\beta)$ stays for the angle formed by $\alpha$ and $\beta$). 
This is possible (for the values listed above) in the following three cases only: 
$$\angle(\alpha,\alpha_1)=\angle(\alpha,\alpha_2)=\pi/4,\ \pi/8 \ \text{or} \ \pi/10.$$
This implies that if
$f_1$ and $f_2$ are facets of $P_1$ forming an angle $\pi/(2k+1)$ for some $k\in \Z$ then neither $f_1$ nor $f_2$ coincides with $\alpha$.
So, both $f_1$ and $f_2$ are facets of $P$.

Now, since $\Sigma'_{\mathrm{odd}}=\Sigma'$, all nodes of $\Sigma'$ correspond to facets of $P$, which proves the lemma.
\end{proof}

\begin{example}
Let $P$ be the $5$-dimensional polytope with $8$ facets shown on Fig.~\ref{ex-subd}.
We use Lemma~\ref{dis_subd} to show that $P$ can not be dissected. 

\begin{figure}[!h]
\begin{center}
\epsfig{file=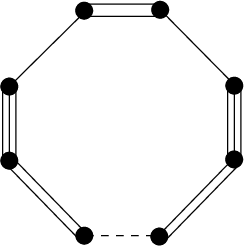,width=0.16045\linewidth}
\caption{ $5$-dimensional polytope with $8$ facets }
\label{ex-subd}
\end{center}
\end{figure}

Suppose that $P$ is dissected into some polytopes $P_1$ and $P_2$. Then by Property~$\mathrm{(I)}$
of Lemma~\ref{pr_d} each of $P_1$ and $P_2$ has at most 8 facets. Hence, both $P_1$ and $P_2$ are in the list of currently known polytopes,
and we may apply  Lemma~\ref{dis_subd}. 
Notice that each diagram of a $5$-polytope with at most 8 facets contains a subdiagram $\Sigma'$ such that $\Sigma'_{\mathrm{odd}}$ is connected
and $\Sigma'$ contains at least 3 simple edges. On the other hand, $\Sigma(P)$ contains only two simple edges. 
The contradiction to Lemma~\ref{dis_subd} shows that the dissection is impossible.
\end{example}

\subsection{Proof of Theorem~\ref{ess-t}}
\label{ess_proof}
We need to show that polytopes listed in Table~\ref{ess-list} are essential and that all the other known polytopes (except two listed in  Question~\ref{qop})
could be obtained from these ones.

\medskip

\noindent
{\bf 1.}
First, we show how to glue the remaining known polytopes from those listed in Table~\ref{ess-list} and Question~\ref{qop}. 

{\bf In dimensions 6--8} there is a unique known polytope not listed above (see~\cite[Table~4.9]{T}), it is the $6$-polytope shown in Fig.~\ref{rest6}, left. 
This polytope is doubled polytope from Fig.~\ref{rest6}, right, where the latter should be reflected in the facet marked white. 

\begin{figure}[!h]
\begin{center}
\psfrag{10}{\scriptsize $10$}
\epsfig{file=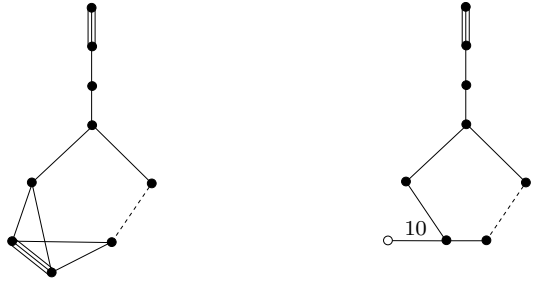,width=0.545\linewidth}
\caption{$6$-polytopes: the left one consists of two copies of the right one}
\label{rest6}
\end{center}
\end{figure}

{\bf In dimension 5} all the remaining known polytopes are once truncated simplices (i.e., simplicial prisms) or twice truncated simplices 
(see \cite{K} and \cite[Table~4.10]{T}). 
It is shown in~\cite[page~20]{T} that if a simplex is not truncated in orthogonal way then it is not essential. However, one of the simplices truncated in orthogonal way 
(i.e., from the list~\cite[Table~A.5]{S}) is not essential either: the polytope shown in Fig.~\ref{rest5}, left, consists of two copies of the prism to the right, 
the copy to glue may be obtained by reflecting in the facet marked white.

\begin{figure}[!h]
\begin{center}
\epsfig{file=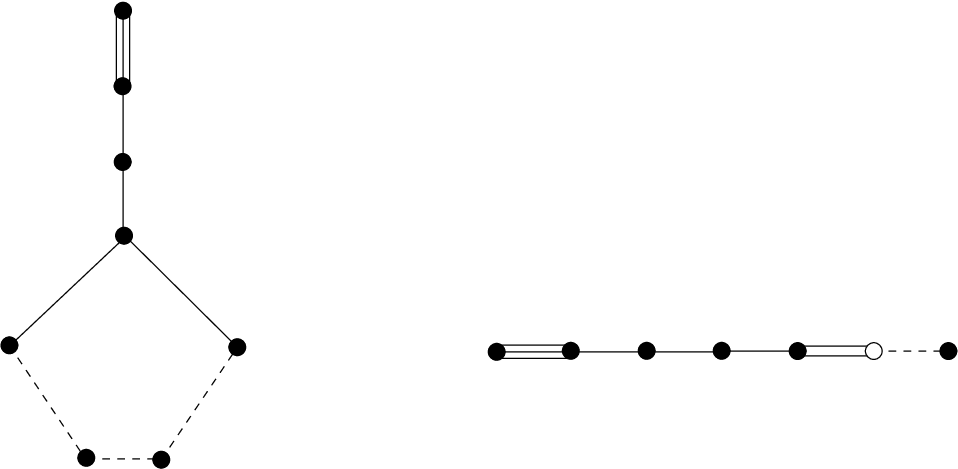,width=0.545\linewidth}
\caption{$5$-polytopes: the left one consists of two copies of the prism to the right}
\label{rest5}
\end{center}
\end{figure}
       
{\bf In dimension 4} the situation is more involved due to larger number of polytopes and combinatorial types. The simplex not listed in  Table~\ref{ess-list} ($[5,3,3^{1,1}]$ in the notation of~\cite{JKRT1}) can be glued according to~\cite[Section~4]{JKRT2} (it is a double of the simplex with both triple and double edges in the diagram). 
All the prisms with one ``right'' base are contained in Table~\ref{ess-list}. The $5$ missed Esselmann polytopes~\cite[Theorem~1.1]{Ess2} are fundamental chambers for subgroups of index $2$ or $4$ of the groups generated by reflections in the facets of two Esselmann polytopes listed in the table. 

We are left to consider $4$-polytopes with at least $7$ facets. The twice truncated simplices can be processed as in dimension $5$. The only twice truncated simplex truncated in orthogonal way (i.e., from the list~\cite{S}) that is not essential is shown on Fig.~\ref{rest4t} to the left, it can be tessellated by $3$ copies of the prism shown 
on Fig.~\ref{rest4t} to the right (the tessellation can be obtained by reflecting the polytope in two facets marked white and gluing all the three copies together).       

\begin{figure}[!h]
\begin{center}
\psfrag{r2}{\footnotesize $\rho_2$}
\epsfig{file=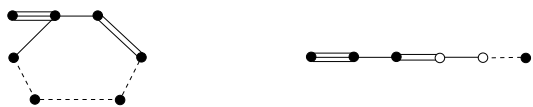,width=0.545\linewidth}
\caption{Twice truncated $4$-simplex to the left consists of three copies of the prism to the right}
\label{rest4t}
\end{center}
\end{figure}

The polytopes with $7$ facets whose diagrams have two multi-multiple edges and one dotted edge clearly tessellate $2$ or $1$ polytopes drawn in the same row 
(see~\cite[Table~4.11]{T}): one should reflect these polytopes in the facets having exactly one non-right angle $\pi/8$ or $\pi/10$. 
The polytopes shown on  Fig.~\ref{rest4s},left and Fig.~\ref{rest4s}, middle ones, consist of two copies of one shown on  Fig.~\ref{rest4s} to the right. 
One shown in the middle of  Fig.~\ref{rest4s} can be obtained by reflecting the right one. However, the polytope shown on  Fig.~\ref{rest4s}, left, does not define 
an index two subgroup. The facet of the initial polytope marked  white is a $3$-simplex whose Coxeter diagram has a symmetry. To get the left polytope we should glue a copy along this facet, not reflected one but symmetric to it.

\begin{figure}[!h]
\begin{center}
\psfrag{r2}{\footnotesize $\rho_2$}
\epsfig{file=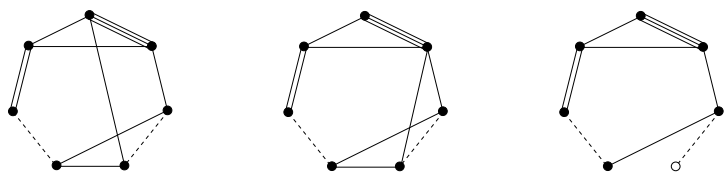,width=0.745\linewidth}
\caption{$4$-polytopes to the left and in the center consist of two copies of the polytope to the right}
\label{rest4s}
\end{center}
\end{figure}

Amongst the three times truncated simplices (they all have 8 facets) only two are not  listed in  Table~\ref{ess-list}: these are the polytopes shown on  
Fig.~\ref{rest4s3} to the left.
Each of these two polytopes can be obtained by gluing of two copies of the corresponding polytope shown on  Fig.~\ref{rest4s3}, right, 
where the latter should be glued by the facet marked white. 

\begin{figure}[!h]
\begin{center}
\psfrag{r2}{\footnotesize $\rho_2$}
\epsfig{file=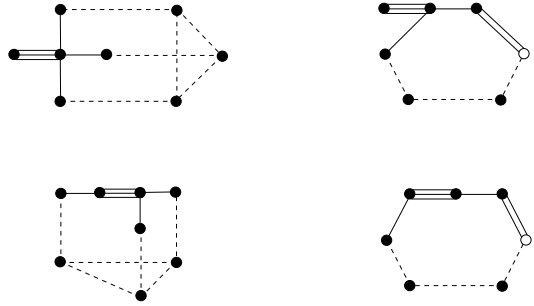,width=0.545\linewidth}
\caption{Each of two $4$-polytopes to the left consists of two copies of the corresponding polytope to the right}
\label{rest4s3}
\end{center}
\end{figure}

\bigskip

\noindent
{\bf 2.}
To prove that all the polytopes listed in the Table~\ref{ess-list} are essential we use Lemmas~\ref{pr_s},~\ref{pr_d}, and~\ref{dis_subd}.
As above, we go through all combinatorial types by increasing complexity and decreasing dimension. Notice that we should also consider decompositions of polytopes into non-essential polytopes.

The main consideration is that for any polytope in   the Table~\ref{ess-list} we know the complete list of polytopes of the same dimension
with smaller number of facets or with the same combinatorics. Indeed, we know all $d$-polytopes with at most $d+3$ facets; it is shown in~\cite{dplus4} 
that there exists a unique $7$-polytope with $7+4=11$ facets, and all the other polytopes from  Table~\ref{ess-list} are three times truncated simplices,
the complete list of which is also known due to~\cite{S}.

So, we have finitely many possibilities for the parts of dissection or for a fundamental chamber of a subgroup.
Lemmas~\ref{pr_s} and~\ref{pr_d} hugely reduce the number of possibilities.

{\bf In  dimensions 6--8} nothing is left after applying these lemmas. 

{\bf In dimension 5}, using Lemmas~\ref{pr_s} and~\ref{pr_d} together with  Lemma~\ref{dis_subd} 
we reduce the problem to the following: is the group generated by the polytope shown on Fig.~\ref{5left} maximal or not?
Denote this polytope by $P_5$.
To see that $G_{P_5}$ is maximal, we compare the distances between the disjoint facets in $P_5$ and in all possible candidates to generate a larger group.
Namely, each of the polytopes satisfying the conditions of Lemma~\ref{pr_s}
has a pair of disjoint facets at larger distance one from another than the maximal of the three distances in $P_5$, so, none of the candidates can be embedded in $P_5$.

\begin{figure}[!h]
\begin{center}
\epsfig{file=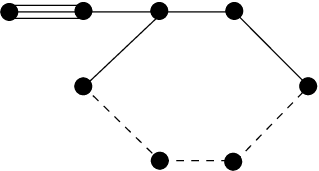,width=0.203045\linewidth}
\caption{The polytope in dimension 5.}
\label{5left}
\end{center}
\end{figure}

{\bf In dimension 4}, we have many possibilities to consider, but we always can compute volumes of polytopes applying  
Poincar\'e formula 
(see e.g.~\cite[Part I, Chapter 7]{29})
and then we can use properties (4) and $\mathrm{(IV)}$ of Lemmas~\ref{pr_s} and~\ref{pr_d}. 
 Applying  Lemmas~\ref{pr_s} and~\ref{pr_d} together with  Lemma~\ref{dis_subd}, we reduce the problem to the following two questions:
\begin{itemize}
\item if $G_{P_4^1}$ is a subgroup of index 2 of  $G_{P_4^2}$?
\item  if $G_{P_4^3}$ is a subgroup of index 3 of  $G_{P_4^4}$? 
\end{itemize}
where $P_4^1$, $P_4^2$, $P_4^3$ and  $P_4^4$ are the polytopes shown on Fig.~\ref{4left}.

\begin{figure}[!h]
\psfrag{P1}{$P_4^1$}
\psfrag{P2}{$P_4^2$}
\psfrag{P3}{$P_4^3$}
\psfrag{P4}{$P_4^4$}
\psfrag{8}{\small $8$}
\psfrag{1}{\scriptsize $v_1$}
\psfrag{2}{\scriptsize $v_2$}
\psfrag{3}{\scriptsize $v_3$}
\psfrag{4}{\scriptsize $v_4$}
\psfrag{5}{\scriptsize $v_5$}
\psfrag{6}{\scriptsize $v_6$}
\epsfig{file=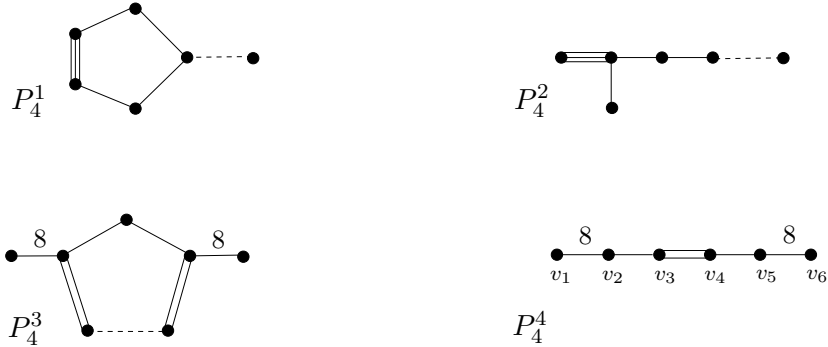,width=0.845\linewidth}
\caption{Some polytopes in dimension 4.}
\label{4left}
\end{figure}

To answer the first question note, that in case of a subgroup of index 2 we actually deal with a dissection, 
so, we apply Lemma~\ref{dis_subd} to see that the answer is negative.

To answer the second question note, that if $G_{P_4^3}$ is a subgroup of  $G_{P_4^4}$, then the neighborhood of the vertex of $P_4^3$ corresponding to the subdiagram $2G_2^{(8)}\subset \Sigma(P_4^3)$ is tiled by the neighborhood of the vertex of $P_4^4$ corresponding to 
$2G_2^{(8)}\subset \Sigma(P_4^4)$, which implies that the facets of $P_4^4$ corresponding to the nodes $v_1$, $v_2$ , $v_5$, and $v_6$  are the facets of $P_4^3$,
while the remaining two facets are not. So, two copies of $P_4^4$  should be attached by the facets $v_3$ and $v_4$,
however this does not lead to a Coxeter polytope (the polytope obtained has an angle $2\pi/3$).

This completes the proof of the theorem.

\clearpage
\pagebreak

\begin{table}[!h]
\caption{Known essential polytopes of dimension at least $4$}
\label{ess-list}
\vspace{6pt}
\underline{$d=4$}:
\end{table}

\vspace{1pt}

\noindent
\u{Simplices:}\\

{\hfill 
\psfrag{3,4,5}{\tiny $3,\!4,\!5$}
\epsfig{file=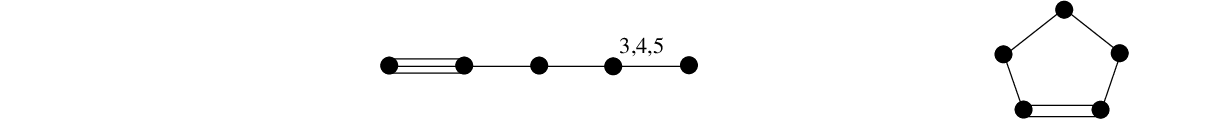,width=0.9\linewidth}
}

\vspace{20pt}

\noindent
\u{Esselmann polytopes:}\\

{\hfill
\psfrag{10}{\scriptsize $10$}
\psfrag{8}{\scriptsize $8$}
\epsfig{file=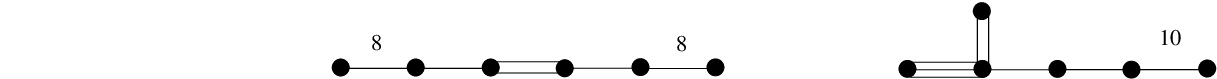,width=0.9\linewidth}
}

\vspace{22pt}

\noindent
\u{Simplicial prisms:}\\
\vspace{4pt}

{\hfill
\epsfig{file=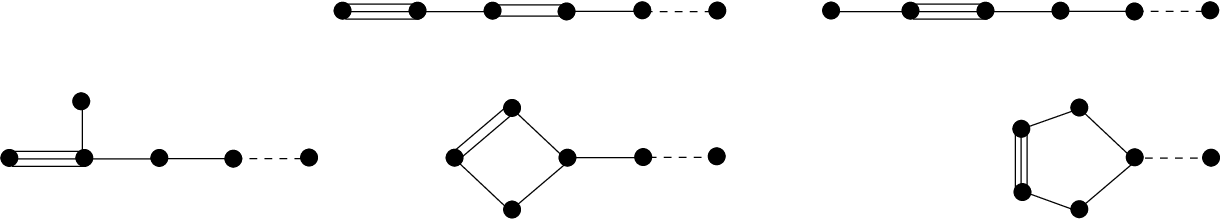,width=0.9\linewidth}
}

\vspace{22pt}

\noindent
\u{$4$-polytopes with 7 facets:}\\
\vspace{5pt}

{\hfill
\psfrag{10}{\scriptsize $10$}
\psfrag{8}{\scriptsize $8$}
\epsfig{file=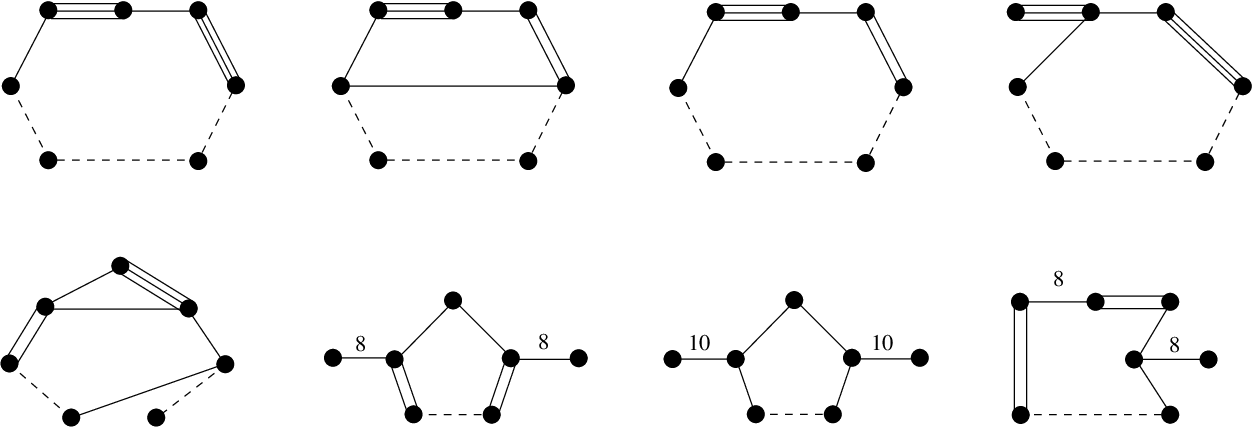,width=0.9\linewidth}
}

\vspace{22pt}

\noindent
\u{Three times truncated simplices:}\\
\vspace{5pt}

{\hfill
\psfrag{3,4,5}{\tiny $3,\!4,\!5$}
\psfrag{4,5}{\tiny $4,\!5$}
\epsfig{file=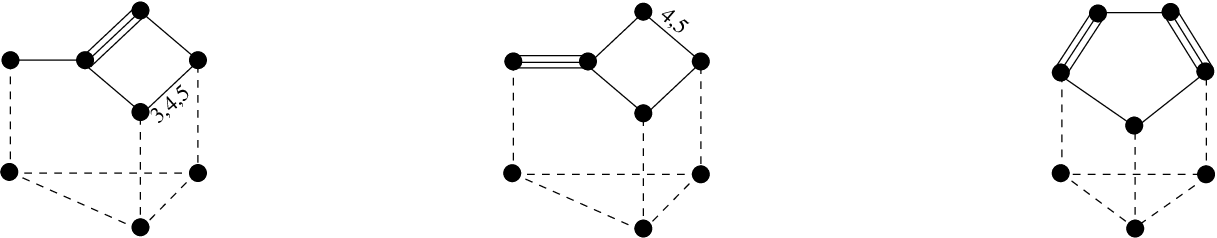,width=0.9\linewidth}
}


\pagebreak
\clearpage

\addtocounter{table}{-1}

\begin{table}[!h]
\caption{Cont.}
\label{ess-list5-8}
\vspace{6pt}
\underline{$d=5$}:
\end{table}

\vspace{8pt}

\noindent
\u{Simplicial prisms:}
\vspace{5pt}

{\hfill
\psfrag{3,4}{\tiny $3,\!4$}
\epsfig{file=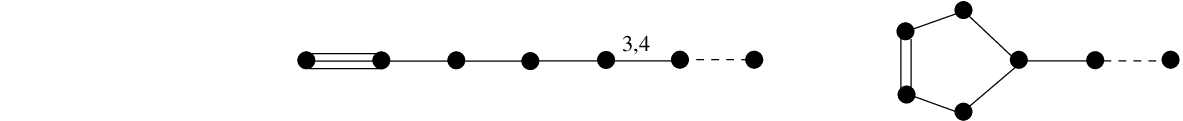,width=0.9\linewidth}
}

\vspace{10pt}
\noindent
\u{$5$-polytopes with 8 facets:}

\vspace{5pt}

{\hfill
\epsfig{file=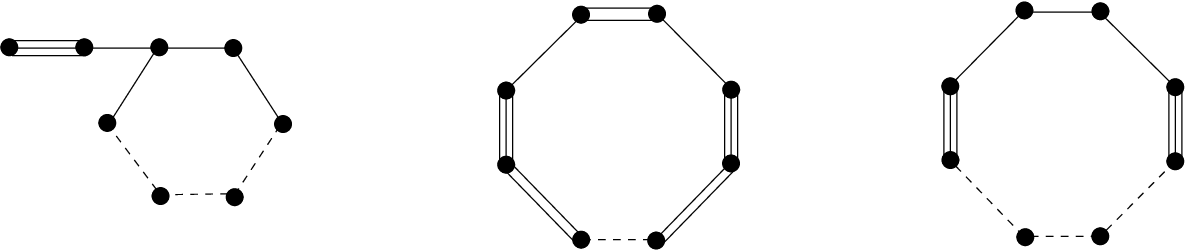,width=0.9\linewidth}
}

\vspace{10pt}
\noindent
\u{Three times truncated simplex:}
\vspace{5pt}

{\hfill
\qquad \quad
\epsfig{file=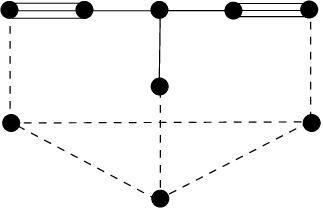,width=0.22845\linewidth}
\hfill
}
\vspace{10pt}

\vspace{25pt}
{\hfill \underline{$d=6$:} \hfill}
\vspace{-25pt}
\begin{center}
\begin{tabular}{cc}
\epsfig{file=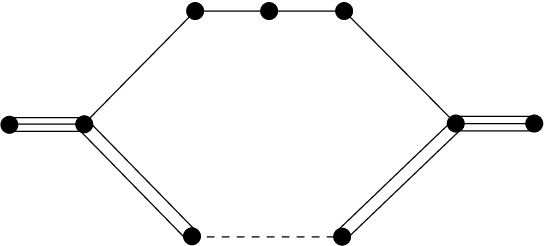,width=0.36\linewidth}&
\psfrag{10}{\scriptsize $10$}
\epsfig{file=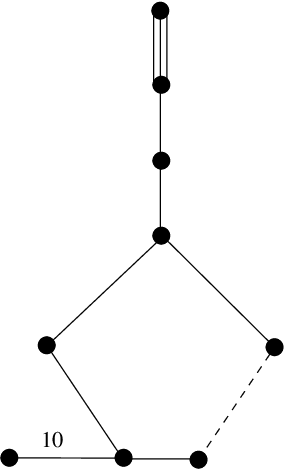,width=0.17\linewidth}\\
\\
\\
\\
\\
\underline{$d=7$}: & \underline{$d=8$}: \\
\\
\epsfig{file=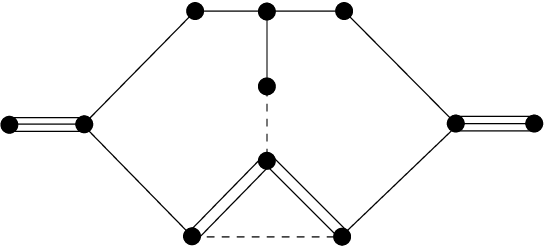,width=0.36\linewidth} & 
\raisebox{40pt}{\epsfig{file=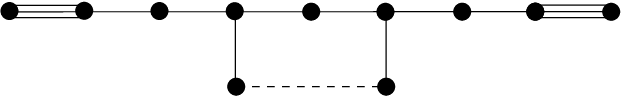,width=0.44\linewidth} }
\\
\end{tabular}
\end{center}

\pagebreak

\end{document}